\documentclass[a4paper, 12pt]{amsart}
\usepackage{a4wide}

\usepackage{amsmath}
\usepackage{amsthm}
\usepackage{amsfonts}
\usepackage{stmaryrd}
\usepackage{amssymb}
\usepackage{sseq}
\usepackage{array}
\usepackage[all]{xy}
\usepackage{color}
\definecolor{darkblue}{rgb}{0,0,0.9}
\definecolor{lightblue}{rgb}{.5,.5,.9}
\usepackage[pdftex,breaklinks,colorlinks,filecolor=blue,linkcolor=blue,citecolor=blue,urlcolor=blue,hypertexnames=true,plainpages=false]{hyperref}

\usepackage{mathrsfs}

\usepackage{tikz}
\usepackage{etoolbox}
\usepackage{comment}
\usepackage{enumitem}
\numberwithin{equation}{section}
\newcommand\Q{\mathbb{Q}}
\newcommand\Z{\mathbb{Z}}

\newcommand{\M}{\mathcal{M}}

\newcommand\tensor{\otimes}

\renewcommand{\Im}{\textup{Im}}

\newcommand{\bocj}[1]{\beta^{\Z/\!2^{#1}}}
\newcommand{\bocqz}{\beta^{\Q/\Z}}

\newcommand{\wt}{\widetilde}

\newcommand{\Sq}{Sq}

\theoremstyle{plain}
\newtheorem{theorem}{Theorem}[section]
\newtheorem{lemma}[theorem]{Lemma}
\newtheorem{corollary}[theorem]{Corollary}
\newtheorem{proposition}[theorem]{Proposition}

\theoremstyle{definition}

\theoremstyle{remark}
\newtheorem{remark}[theorem]{Remark}

\usepackage{color}

\advance \topmargin by -\headheight
\advance \topmargin by -\headsep
     
\evensidemargin \oddsidemargin
\makeatletter
\let\uppercasenonmath\@gobble% disables title uppercase
\makeatother

\title{A Bilinear Form for Spin$^c$ Manifolds}

\author[Huijun Yang]{Huijun Yang}
\address{School of Mathematics and Statistics, Henan University, Kaifeng 475004, Henan, China}
\email{yhj@amss.ac.cn}

\thanks{This work was supported by the Natural Science Foundation of Henan (No. 242300421380).}


\begin{document}
\maketitle

\begin{abstract}
Let $M$ be a closed oriented spin$^{c}$ manifold of dimension $(8n {+} 2)$ with fundamental class $[M]$, and let $\rho_{2} \colon H^{4n}(M; \Z) \rightarrow H^{4n}(M; \Z/2)$ denote the $\bmod ~ 2$ reduction homomorphism.
For any torsion class $t \in H^{4n}(M;\Z)$, we establish the identity
\[ \langle \rho_2(t) \cdot Sq^2 \rho_2 (t), [M] \rangle = \langle \rho_2 (t) \cdot Sq^2 v_{4n}(M), [M]\rangle, \]
where $Sq^2$ is the Steenrod square, $v_{4n}(M)$ is the $4n$-th Wu class of $M$, $ x\cdot y$ denotes the cup product of $x$ and $y$, and $\langle \cdot ~, ~\cdot \rangle$ denotes the Kronecker product.
This result generalizes the work of Landweber and Stong from spin to spin$^c$ manifolds. 

As an application, let $\beta^{\Z/2} \colon H^{4n+2}(M; \Z/2) \to H^{4n+3}(M; \Z)$ be the Bockstein homomorphism associated to the short exact sequence of coefficients 
$\Z \xrightarrow{\times 2} \Z \to \Z/2$. We deduce that $\beta^{\Z/2}(Sq^2 v_{4n}(M)) = 0$, and consequently, $Sq^3 v_{4n}(M) = 0$, for any closed oriented spin$^{c}$ manifold $M$ with $\dim M \le 8n{+}1$. 
%As applications, the rank of this form on torsion classes gives rise to a cobordism invariant.
\end{abstract}

\section{Introduction} 
\label{s:intro}
%%%%%%%%%%%%%%%%%%%%%%%%%%%%%%%%%%%%%%%%%%%%%%%

%\subsection{Statement of results}
%We establish necessary notation.
Let $X$ be a $CW$-complex.
Throughout this paper, unless specified otherwise (such as Section \ref{s:main}), $H^{\ast}(X)$ denotes its integral cohomology ring.
Let $TH^*(X)$ denote the torsion subgroup of $H^*(X)$, $\Sq^{k} \colon H^{i}(X; \Z/2) \rightarrow H^{i+k}(X; \Z/2)$ the $k$-th Steenrod square,  
%$H^*(X)$ denotes the integral cohomology ring of $X$,
$\rho_{2} \colon H^*(X) \to H^*(X; \Z/2)$ the mod~$2$ reduction homomorphism. 
The short exact coefficient sequence
$0 \to \Z \xrightarrow{\times 2} \Z \to \Z/2 \to 0$ induces the associated Bockstein long exact sequence:
\begin{equation}\label{eq:bseq}
\cdots \rightarrow H^{i}(X) \xrightarrow{~\times 2~} H^{i}(X)\xrightarrow{~\rho_{2}~} 
H^{i}(X;\Z/2)\xrightarrow{~\bocj{}~} H^{i+1}(X) \to \cdots
\end{equation}
where $\bocj{} \colon H^{i}(X;\Z/2) \to H^{i+1}(X)$ is the Bockstein homomorphism.
%The $k$th Steenrod square is 
%$\Sq^{k} \colon H^{i}(X; \Z/2) \rightarrow H^{i+k}(X; \Z/2)$. 
%$TH^*(X)$ denotes the torsion subgroup of $H^*(X)$.

Unless otherwise stated, all manifolds considered in this paper are assumed to be smooth, closed, and oriented.
For a manifold $M$, we denote by $w_{i}(M)$ and $v_{i}(M)$ its $i$-th Stiefel-Whitney class and Wu class, respectively, by $[M]$ its fundamental class, and by $\langle \cdot ~, ~\cdot \rangle$ the Kronecker product.

For a $(4n{+}1)$-dimensional manifold $M$, the symmetric bilinear form 
\begin{align*}
H^{2n}(M;\Z/2) \times H^{2n}(M;\Z/2) \to \Z/2: \quad (x, y) \mapsto \langle x \cdot Sq^1 y, [M]\rangle,
\end{align*}
where $x \cdot y$ denotes the cup product of $x$ and $y$, has been investigated by Browder \cite{br62} and Lusztig, Milnor and Peterson \cite{lmp69}.
Browder \cite[Lemma 5]{br62} established the identity
\[ \langle x \cdot \Sq^{1}  x , [M] \rangle = \langle x \cdot \Sq^{1}  v_{2n}(M) , [M] \rangle, ~\text{for any $x \in H^{2n}(M; \Z/2)$}.\]
%which holds for any $x \in H^{2n}(M; \Z/2)$.
Building on this identity, Lusztig, Milnor and Peterson \cite{lmp69} prove that the difference between the semi-characteristics over $\Z/2$ and $\Q$ equals a specific Stiefel-Whitney number, thereby clarifying when the semi-characteristic is independent of the coefficient field.

For an $(8n{+}2)$-dimensional manifold $M$, Landweber and Stong \cite{ls87} consider the bilinear form 
\begin{align*}
H^{4n}(M) \times H^{4n}(M) \to \Z/2: \quad (x, y) \mapsto \langle \rho_2(x) \cdot Sq^2 \rho_2(y), [M]\rangle.
\end{align*}
%They obtained an analogous result of the identity above for spin manifolds.
Assuming $M$ is spin (i.e., $w_2(M) = 0$),
they established the identity
%They proved that for any $(8n{+}2)$-dimensional spin manifold $M$ (i.e., $w_2(M) = 0$) and any $x \in H^{4n}(M)$,
\[ \langle \rho_{2}(x) \cdot \Sq^{2} \rho_{2} ( x ), [M] \rangle = \langle \rho_{2} ( x ) \cdot \Sq^{2}  v_{4n}(M) , [M] \rangle,~\text{for any $x \in H^{4n}(M)$},\]
which is directly analogous to the identity obtained by Browder \cite{br62}, mentioned above.
As a corollary, they deduced that $Sq^3 v_{4n}(M)=0$ for any $(8n{+}2)$-dimensional spin manifold $M$.
It is worth noting that, related to this bilinear form, Fang and Pan \cite{fp04} defined secondary Brown-Kervaire quadratic forms for $\Phi$-oriented manifolds, which they used to provide a complete classification $(n{-}2)$-connected $2n$-dimensional $\pi$-manifolds up to homotopy equivalence and homeomorphism for $n\ge 4$ when $n{+}2$ is not a power of $2$.

In this paper, we generalize the result of Landweber and Stong \cite{ls87} to spin$^c$ manifolds.
For an $(8n{+}2)$-dimensional manifold $M$, we consider the bilinear form 
\begin{align*}
TH^{4n}(M) \times TH^{4n}(M) \to \Z/2: \quad (x, y) \mapsto \langle \rho_2(x) \cdot Sq^2 \rho_2(y), [M]\rangle,
\end{align*}
defined on the torsion subgroup.
%we generalize the result of Landweber and Stong to spin$^c$ manifolds.
One of our main results is the following theorem.
\begin{theorem}\label{thm:equiv}
The following two statements are equivalent:

1) for any $(8n{+}2)$-dimensional spin manifold $M$, and any $x \in H^{4n}(M)$,
\[ \langle \rho_{2}(x) \cdot \Sq^{2} \rho_{2} ( x ), [M] \rangle = \langle \rho_{2} ( x ) \cdot \Sq^{2}  v_{4n}(M) , [M] \rangle.\]

2) for any $(8n{+}2)$-dimensional spin manifold $M$, and any torsion class $t \in TH^{4n}(M)$,
\[ \langle \rho_{2}(t) \cdot \Sq^{2} \rho_{2} ( t ), [M] \rangle = \langle \rho_{2} ( t ) \cdot \Sq^{2}  v_{4n}(M) , [M] \rangle.  \]
\qed
\end{theorem}

Furthermore, in the sense of Theorem \ref{thm:equiv}, the result of Landweber and Stong \cite[Proposition 1.1]{ls87} generalizes as follows.
\begin{theorem}\label{thm:main}
For any $(8n {+} 2)$-dimensional spin$^{c}$ manifold $M$ and any torsion class $t \in TH^{4n}(M)$,
\begin{equation*}
\langle \rho_{2} ( t ) \cdot \Sq^{2} \rho_{2} ( t ) , [M] \rangle = \langle \rho_{2}( t ) \cdot \Sq^{2} v_{4n}(M) , [M] \rangle.
\end{equation*}
\end{theorem}

\begin{remark}
For $n=1$, this result has been proved by Crowley and the author \cite[Theorem 2.2]{cy20}.
\end{remark}

\begin{remark}
It follows from Landweber and Stong \cite[p. 637]{ls87} that there is no class $y \in H^{4n+2}(B\mathrm{Spin}^c; \Z/2)$ with $n> 0$ such that the identity
\[ \langle x \cdot Sq^2 x, [M] \rangle = \langle x \cdot \tau_{M}^{\ast} (y), [M] \rangle \]
holds for all $(8n{+}2)$-dimensional spin$^c$ manifolds $M$ and all $x \in H^{4n}(M;\Z/2)$. 
It is therefore natural to ask whether the identity in Theorem \ref{thm:main} holds for any $x \in H^{4n}(M)$, not merely for torsion classes.
Unfortunately, this question remains open.
\end{remark}

\begin{remark}
One may also ask whether there exists a universal class $y \in H^{n+1}(B\mathrm{Spin}^c; \Z/2)$ such that 
\[ \langle \rho_{2} ( t ) \cdot \Sq^{2} \rho_{2} ( t ) , [M] \rangle = \langle \rho_{2}( x ) \cdot \tau_{M}^{\ast} (y) , [M] \rangle\]
holds for any $2n$-dimensional spin$^c$ manifold $M$ and any $t \in TH^{n-1}(M)$.
For $n \le 3$, the answer is affirmative, and one may take $y = 0$.
For $n=4k$, $k \ge 1$, it follows from Landweber and Stong \cite[p. 638]{ls87} that the answer is negative.
The cases $n=4k+2$ and $n=4k+3$ for $k \ge 1$ remain unresolved.
\end{remark}

As an application of our main theorem, we obtain the following corollary. 
\begin{corollary}\label{coro:bsqv}
For any $(8n{+}1)$-dimensional spin$^c$ manifold $M$, we have 
\[ \beta^{\Z/2} ( Sq^2 v_{4n}(M) ) = 0,\]
and consequently,  $Sq^3 v_{4n}(M)  = 0$.
\end{corollary}

\begin{remark}
It follows immediately from Corollary \ref{coro:bsqv} that 
$ \beta^{\Z/2} ( Sq^2 v_{4n}(M) )= 0$,
and hence $Sq^3 v_{4n}(M) = 0$, for any spin$^c$ manifold $M$ with $\dim M \le 8n+1$.
\end{remark}

\begin{remark}
One can see from the proof of Theorem \ref{thm:main} that, with the exception of the case $n = 2$, $Sq^3 v_{4n}$ is the only nonzero class of dimension $4n+3$ that vanishes on every spin$^c$ manifold of dimension $\le 8n+1$.
\end{remark}

\begin{remark}
Diaconescu, Moore and Witten \cite[Appendix D]{dmw02} proved that there exists a spin $10$-manifold $M$ with $\beta^{\Z/2}(Sq^2 v_4(M)) \neq 0$.
\end{remark}

\begin{remark}
Wilson \cite{wi73}  and Landweber and Stong \cite{ls87} both demonstrated that $Sq^3 v_{4n} =0$ for every spin manifold of dimension $8n+2$.
However, we cannot extend this conclusion to spin$^c$ manifolds.
In fact, we conjecture that $Sq^3 v_{4n} \ne 0$, and hence $\beta^{\Z/2} ( Sq^2 v_{4n} ) \neq 0$, for some $(8n{+}2)$-dimensional spin$^c$ manifold.
\end{remark}

For an $(8n {+} 2)$-dimensional spin$^{c}$ manifold $M$, let $TV^{4n}(M; \Z/2)$ denote the subspace of $H^{4n}(M; \Z/2)$ spanned by $\rho_2 ( TH^{4n}(M) )$ and $v_{4n}(M)$. 
Consider the bilinear form
\begin{equation*}
[~,~] \colon TV^{4n}(M; \Z/2) \times TV^{4n}(M; \Z/2) \rightarrow \Z/2
\end{equation*}
defined by $[x, y] = \langle x \cdot \Sq^{2} y , [M]\rangle$.
Since $Sq^1 v_{4n}(M) =0 $ by Lemma \ref{lem:v2k} (Subsection \ref{ss:pfmain}), and since $v_{2}(M) = w_{2}(M) \in \rho_{2}( H^{2}(M) )$, the definition of the Wu class implies that the bilinear form $[~,~]$ is symmetric.

\begin{corollary}\label{coro:rank}
For an $(8n {+} 2)$-dimensional spin$^{c}$ manifold $M$, the expression
\[ \langle ( w_4(M)+w_2^2(M) ) \cdot w_{8n-2}(M), [M] \rangle = \langle v_{4n}(M) \cdot Sq^2 v_{4n}(M), [M]\rangle \]
is equal to the mod $2$ rank of the bilinear form $[~, ~]$ on $TV^{4n}(M; \Z/2)$.
\end{corollary}

\begin{proof}
It follows directly from the proof of the theorem in \cite{lmp69} that 
\[ \langle v_{4n}(M) \cdot Sq^2 v_{4n}(M), [M]\rangle \]
equals the mod $2$ rank of the bilinear form $[~, ~]$.
To complete the proof, we verify the stated equality.
%\[ \langle ( w_4(M)+w_2^2(M) ) \cdot w_{8n-2}(M), [M] \rangle = \langle v_{4n}(M) \cdot Sq^2 v_{4n}(M), [M]\rangle. \]
Since $v_{odd}(M) = 0$ and $v_j(M) = 0$ for $j > 4n{+}1$, Wu's formula  (cf. \cite[p. 132, Theorem 11.14]{ms74}) 
\begin{equation} \label{eq:wu}
w_k(M) = \Sigma \Sq^i v_{k-i}(M)
\end{equation}
implies that 
$v_4(M) = w_4(M) + w^2_2(M)$ and $w_{8n-2}(M) = Sq^{4n-2}v_{4n}(M)$.
Therefore, 
\begin{align*}
( w_4(M) + w_2^2(M) ) \cdot w_{8n-2}(M) = v_4(M) \cdot Sq^{4n-2} v_{4n}(M) = Sq^4 Sq^{4n-2} v_{4n}(M).
\end{align*}
Furthermore, by the Adem relation \eqref{eq:adem} below, $Sq^4 Sq^{4n-2} = \binom{4n-3}{4} \Sq^{4n+2} + \Sq^{4n} \Sq^2 $.
Since $\Sq^{4n+2} v_{4n}(M) = 0$, we obtain
\begin{align*}
( w_4(M) + w_2^2(M) ) \cdot w_{8n-2}(M)  = \Sq^{4n} \Sq^2 v_{4n}(M) = v_{4n}(M) \cdot \Sq^2 v_{4n}(M),
\end{align*}
which completes the proof.
\end{proof}

The paper is organized as follows.
Section \ref{s:equiv} provides necessary notation and the proof of Theorem \ref{thm:equiv}.
The proof of Theorem \ref{thm:main} is more complicated and Sections \ref{s:exis}-\ref{s:main} are devoted to it.
In Section \ref{s:exis} we show that there exists exists a class $\Theta \in H^{4n}(B\mathrm{Spin}^c; \Z/2)$ such that 
\[ \langle \rho_2(t) \cdot \Sq^2 \rho_2 (t), [M]\rangle = \langle \rho_2(t) \cdot \tau_{M}^{\ast}(\Theta), [M] \rangle \] 
holds for any $(8n{+}2)$-dimensional spin$^c$ manifold $M$ and any torsion class $t \in TH^{4n}(M)$, where $\tau_{M} \colon M \to B\mathrm{Spin}^c$ classifies the stable tangent bundle of $M$. 
Section \ref{s:prop} describes some elementary properties of $\Theta$. 
Finally, in Section \ref{s:main}, based on computations of reduced spin$^c$ bordism groups of $C_{\Psi}$ arising from the cofibration \eqref{cof:Psi}, the class $\Theta$ is uniquely determined.

% {s:intro}
%%%%%%%%%%%%%%%%%%%%%%%%%%%%%%%%%%%%%%%%%%%%%%%

%\section{Preliminary}
%\label{s:pre}
%%%%%%%%%%%%%%%%%%%%%%%%%%%%%%%%%%%%%%%%%%%%%%%%%%%%%%%%%%%

%{s:pre}
%%%%%%%%%%%%%%%%%%%%%%%%%%%%%%%%%%%%%%%%%%%%%%%%%%%%%%%%%%%%

\section{Proof of Theorem \ref{thm:equiv}}
\label{s:equiv}
%%%%%%%%%%%%%%%%%%%%%%%%%%%%%%%%%%%%%%%%%%%%%%%%%%%%%%%%%%%
This section is devoted to the proof of Theorem \ref{thm:equiv}.

We begin by establishing the necessary notation.
For any $CW$-complex $X$, consider the Bockstein long exact sequence associated to the coefficient sequence 
$ \Z \to \Q \to \Q/\Z$:
\begin{align}\label{boc:q/z}
\cdots \to H^n(X; \Q) \xrightarrow{\rho} H^n(X;\Q/\Z) \xrightarrow{\beta^{\Q/\Z}} H^{n+1}(X) \to H^{n+1}(X;\Q)\to \cdots,
\end{align}
where $\beta^{\Q/\Z}$ denotes the Bockstein homomorphism.

Let $K(G, n)$ denote the Eilenberg-MacLane space of type $(G, n)$, and let
$l_{n} \in H^{n}(K(\Z, n))$ and $l_{n}^{T} \in H^{n}(K(\Q/\Z, n); \Q/\Z)$
be the fundamental classes. 
By the Brown representation theorem (cf. \cite[p. 182, Theorem 10.21]{sw02b}), there exists a Bockstein map
\begin{equation}\label{map:beta}
\bar{\beta} \colon K(\Q/\Z, n) \rightarrow K(\Z, n{+}1)
\end{equation}
that corresponds to the Bockstein homomorphism 
\[ \beta^{\Q/\Z} \colon H^{n}(K(\Q/\Z, n); \Q/\Z) \rightarrow H^{n+1}(K(\Q/\Z, n)). \]
%By definition, we have
%\begin{equation}\label{eq:beta}
%\bar{\beta}^{\ast} (i_{n+1}) = \beta^{\Q/\Z} (i_{n}^{T}).
%\end{equation}
For any $x \in H^{n+1}(X)$ and $z \in H^{n}(X; \Q/\Z)$,
we denote by 
$$f_{x} \colon X \rightarrow K(\Z, n+1) ~ (\text{respectively,~} f_{z} \colon X \rightarrow K(\Q/\Z, n))$$
 the maps satisfying $f_{x}^{\ast} (l_{n+1}) = x $ (respectively, $f_{z}^{\ast} (l_{n}^{T}) = z$).
 
Now, suppose $t \in TH^{n+1}(X)$, the torsion subgroup of $H^{n+1}(X)$. The exactness of the Bockstein sequence \eqref{boc:q/z} implies the existence of a class 
$ z \in H^{n}(X; \Q/\Z) $ such that 
\[ \beta^{\Q/\Z} (z) = t. \]
Consequently, by the definition of $\bar{\beta}$, we have 
\begin{equation}\label{eq:betaf}
f_{t} = \bar{\beta} \circ f_{z}.
\end{equation}

For any $CW$-complex $X$, let $\wt{\Omega}_{\ast}^{\mathrm{Spin}} (X)$ denote the reduced spin bordism groups of $X$.
An element of $[N, f] \in \wt{\Omega}^{\mathrm{Spin}}_n(X)$ is represented by a map $f \colon N \to X$ from a closed spin $n$-manifold $N$.

\begin{lemma}\label{lem:qz}
For any positive integer $n$, the induced homomorphism 
\[ \bar{\beta}_{\ast} \colon \wt{\Omega}_{8n+2}^{\mathrm{Spin}}(K(\Q/\Z, 4n{-}1)) \to \wt{\Omega}_{8n+2}^{\mathrm{Spin}}(K(\Z, 4n)) \]
is an isomorphism.
\end{lemma}

\begin{proof}
Let $C_{\bar{\beta}}$ denote the mapping cone of $\bar{\beta}$, which gives rise to the cofibration sequence:
\[ K(\Q/\Z, 4n-1) \xrightarrow{\bar{\beta}} K(\Z, 4n) \to C_{\bar{\beta}}.\]
This sequence induces a long exact sequence in bordism groups: 
\begin{align}
\cdots \to \wt{\Omega}_{8n+3}^{\mathrm{Spin}}(C_{\bar{\beta}}) \to \wt{\Omega}_{8n+2}^{\mathrm{Spin}}(K(\Q/\Z, 4n-1)) \xrightarrow{\bar{\beta}_{\ast}} \wt{\Omega}_{8n+2}^{\mathrm{Spin}}(K(\Z, 4n)) \to \wt{\Omega}_{8n+2}^{\mathrm{Spin}}(C_{\bar{\beta}}) \to \cdots
\end{align}
Thus, to prove the lemma, it suffices to show that the bordism groups $\wt{\Omega}_{8n+3}^{\mathrm{Spin}}(C_{\bar{\beta}})$ and $\wt{\Omega}_{8n+2}^{\mathrm{Spin}}(C_{\bar{\beta}})$ are both trivial.

This conclusion follows from the Atiyah-Hirzebruch spectral sequence for $C_{\bar{\beta}}$: 
\[
\bigoplus \wt{H}_{p}(C_{\bar{\beta}}; \Omega_{q}^{\mathrm{Spin}}) \implies 
\wt{\Omega}_{p+q}^{\mathrm{Spin}}(C_{\bar{\beta}}).
\]
By construction, the integral homology of $C_{\bar{\beta}}$ satisfies
\begin{align}\label{eq:hcb}
H_{\ast}(C_{\bar{\beta}}) \cong H_{\ast}(K(\Z, 4n);\Q).
\end{align}
Furthermore, applying the universal coefficient theorem yields
\begin{align}\label{eq:hk4nq}
H_{\ast}(K(\Z, 4n);\Q) \cong H^{\ast}(K(\Z, 4n);\Q) \cong \Q[x],
\end{align}
where $x \in H^{4n}(K(\Z, 4n);\Q)$ is a generator (cf. Hatcher \cite[p. 550, Proposition 5.21]{hat-bat}).
Since the spin bordism groups $\Omega_q^{\mathrm{Spin}}$ are torsion for $q \not \equiv 0 \mod 4$ (cf. Stong \cite[p. 340, Theorem]{st68b}), Equations \eqref{eq:hcb}, \eqref{eq:hk4nq}, and the universal coefficient theorem together imply that 
\[ \wt{H}_{p}(C_{\bar{\beta}}; \Omega_{q}^{\mathrm{Spin}}) \cong \wt{H}_{p}(C_{\bar{\beta}}; \Z) \otimes_{\Z} \Omega_{q}^{\mathrm{Spin}} = 0 \]
for $p+q = 8n+2$ and $8n+3$. Therefore, $\wt{\Omega}_{8n+3}^{\mathrm{Spin}}(C_{\bar{\beta}}) = \wt{\Omega}_{8n+2}^{\mathrm{Spin}}(C_{\bar{\beta}}) = 0$ and the desired isomorphism follows.
\end{proof}

\begin{proof}[Proof of Theorem \ref{thm:equiv}]
That implication $1) \Rightarrow 2)$ is immediate. 
To prove that $2)$ implies $1)$, we define homomorphisms 
\begin{align*}
\varphi \colon \wt{\Omega}_{8n+2}^{\mathrm{Spin}}(K(\Z, 4n)) \to \Z/2, \\
\phi \colon \wt{\Omega}_{8n+2}^{\mathrm{Spin}}(K(\Z, 4n)) \to \Z/2,
\end{align*}
by 
\begin{align*}
 \varphi([N, f]) & = \langle \rho_2 ( f^{\ast}(l_{4n}) ) \cdot \Sq^{2} \rho_2 ( f^{\ast}(l_{4n}) ), ~[N] \rangle, \\
 \phi([N,f]) & = \langle \rho_2 ( f^{\ast}(l_{4n}) ) \cdot \Sq^{2} v_{4n}(N), ~[N] \rangle,
 \end{align*}
for any bordism class $[N, f]\in  \wt{\Omega}^{\mathrm{Spin}}_{8n+2}(K(\Z, 4n))$ represented by $f\colon N \rightarrow K(\Z, 4n)$.

With the notation as above, for any $(8n{+}2)$-dimensional spin manifold $M$ and any nonzero $x \in H^{4n}(M)$, the pair $(M, f_x)$ determines a bordism class
$ [M, f_x] \in \wt{\Omega}_{8n+2}^{\mathrm{Spin}}(K(\Z, 4n))$. 
Since $\bar{\beta}_{\ast}$ is an isomorphism by Lemma \ref{lem:qz}, there exists a bordism class $[N, f_z] \in \wt{\Omega}_{8n+2}^{\mathrm{Spin}}(K(\Q/\Z, 4n{-}1))$ such that
\begin{align*}
[M, f_x] = \bar{\beta}_{\ast} ( [N, f_z] ) = [N, \bar{\beta} \circ f_z] = [N, f_t],
\end{align*}
where $z \in H^{4n-1}(N; \Q/\Z)$ and $t = \beta^{\Q/\Z} ( z ) \in TH^{4n}(N; \Z)$.
Therefore, by statement $2)$,
\begin{align*}
\langle \rho_2(x) \cdot Sq^2 \rho_2(x), [M] \rangle & = \langle \rho_2 ( f_x^{\ast} (l_{4n}) ) \cdot Sq^2 \rho_2 ( f_x^{\ast} (l_{4n}) ), [M] \rangle \\
& = \varphi([M, f_x]) \\
& = \varphi([N, f_t]) \\
& = \langle \rho_2 ( f_t^{\ast} (l_{4n}) ) \cdot Sq^2 \rho_2 ( f_t^{\ast} (l_{4n}) ), [N] \rangle \\
& = \langle \rho_2 ( t ) \cdot Sq^2 \rho_2 ( t ), [N] \rangle \\
& = \langle \rho_2 (t) \cdot Sq^2 v_{4n}(N), [N] \rangle \\
& = \phi([N, f_t]) \\
& = \phi([M, f_x]) \\
& = \langle \rho_2(x) \cdot Sq^2 v_{4n}(M), [M] \rangle,
\end{align*}
which completes the proof.
\end{proof}

%{s:thm1}
%%%%%%%%%%%%%%%%%%%%%%%%%%%%%%%%%%%%%%%%%%%%%%%%%%%%%%%%%%%%%%

\section{Existence of $\Theta$}  
\label{s:exis}
%%%%%%%%%%%%%%%%%%%%%%%%%%%%%%%%%%%%%%%%%%%%%%%
%The remainder of this paper will be devoted to the proof of Theorem \ref{thm:main}. 
%Denote by $B\mathrm{Spin}^c$ the classifying space of the stable spin$^c$ group. 
%In this section, we will prove the following theorem.

\begin{theorem}\label{thm:theta}
There exists a class $\Theta \in H^{4n+2}(B\mathrm{Spin}^{c};\Z/2)$, such that for any $(8n{+}2)$-dimensional spin$^{c}$ manifold $M$ and for any torsion class $t\in TH^{4n}(M)$, we have
\begin{equation*}
\langle \rho_{2} (t) \cdot \Sq^{2} \rho_{2}(t), ~[M] \rangle = \langle \rho_{2} (t) \cdot \tau_M^{\ast} (\Theta), ~[M] \rangle
\end{equation*}
where $\tau_M \colon M \rightarrow B\mathrm{Spin}^{c}$ classifies the stable tangent bundle of $M$.
\end{theorem}

To prove this theorem, we require some preliminaries.

For any $CW$-complex $X$, denote by $\wt{\Omega}_{\ast}^{\mathrm{Spin}^c} (X)$ the reduced spin$^c$ bordism groups of $X$.
An element of $[N, f] \in \wt{\Omega}^{\mathrm{Spin}^c}_n(X)$ is represented by a map $f \colon N \to X$
from a closed spin$^c$ $n$-manifold $N$.
For any positive integer $i$ and $r$, define a homomorphism
\[ \mathcal{P} \colon \wt{\Omega}_{r+i}^{\mathrm{Spin}^c}( K(\Z, r) ) \to H_i(B\mathrm{Spin}^c) \]
by
\[ \mathcal{P} ([N, f]) = \tau_{N\ast} ( [N] \cap f^{\ast}(l_r) ) \]
for any bordism class $[N, f]\in  \wt{\Omega}^{\mathrm{Spin}^c}_{r+i}(K(\Z, r))$ with $f\colon N \rightarrow K(\Z, r)$.

\begin{lemma}\label{lem:oh}
For any fixed positive integer $i$, and for sufficiently large $r$, the map
\begin{align*}
\mathcal{P} \colon \wt{\Omega}_{i+r}^{\mathrm{Spin}^c}( K(\Z, r) ) \to H_i(B\mathrm{Spin}^c) 
\end{align*}
is an isomorphism.
\end{lemma}

\begin{proof}
This follows from the sequence of isomorphisms:
\begin{align*}
\varinjlim_{r} \wt{\Omega}_{r+i}^{\mathrm{Spin}^{c}}(K(\Z, r)) & \cong  \varinjlim_{s, r} \pi_{r+8s+i}(M\mathrm{Spin}^{c}(8s) \wedge K(\Z, r)) \\
& \cong  \varinjlim_{s} \wt{H}_{8s+i}(M\mathrm{Spin}^{c}(8s)) \\
& \cong \varinjlim_{s} H_{i}(B\mathrm{Spin}^{c}(8s)) \\
& = H_{i}(B\mathrm{Spin}^{c}),
\end{align*}
where $M\mathrm{Spin}^{c}(8s)$ is the Thom space of the classifying bundle over $B\mathrm{Spin}^{c}(8s)$.
The definitions of the isomorphisms involved verify the claim. 
\end{proof}

For any $CW$-complex $X$ and $Y$, denote by $\Sigma X$ the suspension of $X$, and by $\Sigma f \colon \Sigma X \to \Sigma Y$ the suspension of a map $f \colon X \to Y$. 
For any coefficient group $G$, we denote the suspension isomorphisms in cohomology and bordism by
\begin{align*}
& \sigma \colon H^{\ast}(X; G) \rightarrow H^{\ast +1}(\Sigma X; G), \notag \\
& \sigma \colon \Omega_{\ast}^{\mathrm{Spin}^c}(X) \to \Omega_{\ast}^{\mathrm{Spin}^c}(\Sigma X).
\end{align*}
The use of the same symbol $\sigma$ for these isomorphisms should not cause confusion.
We also recall the Freudenthal suspension theorem (see \cite[Corollary 4.24]{hat-bat}):
\begin{lemma}[Freudenthal suspension theorem]
\label{lem:freu}
Suppose that $X$ is an $(n-1)$-connected $CW$ complex.
Then the suspension map $\pi_{i}(X) \to \pi_{i+1}(\Sigma X)$ is an isomorphism for $i < 2n-1$ and a surjection for $i = 2n-1$. 
\end{lemma}

Now, for large $r$, let us consider the following two cofibrations.
\begin{align}
\Sigma^{r-4n}K(\Z, 4n) &\xrightarrow{\psi} K(\Z,r) \xrightarrow{\pi_{\psi}} C_{\psi} ,\label{cof:psi}\\
\Sigma^{r-4n}K(\Q/\Z, 4n-1) &\xrightarrow{\bar{\psi}} K(\Z,r) \xrightarrow{\pi_{\bar{\psi}}} C_{\bar{\psi}}, \label{cof:psib}
\end{align}
where $\psi \colon \Sigma^{r-4n}K(\Z, 4n) \rightarrow K(\Z, r)$ is the map satisfying
$$\psi^{\ast} (l_{r})= \sigma^{r-4n} (l_{4n}),$$
and $\bar{\psi} = \psi \circ \Sigma^{r-4n} \bar{\beta}$ is the composition.
Here, $\sigma^{k}$ denotes the $k$-fold composition of $\sigma$.
By construction, there exists a map 
$$h \colon C_{\bar{\psi}} \rightarrow C_{\psi}$$
such that the cofibrations \eqref{cof:psi} and \eqref{cof:psib} fit into the commutative diagram:
\begin{equation}\label{eq:cof}
\begin{split}
\xymatrix{
\Sigma^{r-4n}K(\Q/\Z, 4n-1) \ar[r]^-{\bar{\psi}} \ar[d]_{\Sigma^{r-4n} \bar{\beta}} & K(\Z, r) \ar@{=}[d] \ar[r]^-{\pi_{\bar{\psi}}}  & C_{\bar{\psi}}  \ar[d]_{ h } \\
\Sigma^{r-4n}K(\Z, 4n) \ar[r]^-{\psi}  & K(\Z, r) \ar[r]^-{\pi_{\psi}}  & C_{\psi} .
}
\end{split}
\end{equation}
Define a homomorphism $\varphi \colon \wt{\Omega}^{\mathrm{Spin}^{c}}_{8n+2}(K(\Z,4n)) \rightarrow \Z/2$ by
\begin{equation*}
\varphi([N, f])=\langle f^{\ast}(l_{4n})\cdot \Sq^{2} f^{\ast}(l_{4n}), ~[N] \rangle.
\end{equation*}
From the commutative diagram \eqref{eq:cof}, Lemma \ref{lem:oh}, and the suspension isomorphism, we obtain the following commutative diagram with exact horizontal sequences:
\begin{equation*}
\begin{split}
\xymatrix{
\cdots \ar[r] & \wt{\Omega}_{r+4n+3}^{\mathrm{Spin}^{c}}(C_{\bar{\psi}}) \ar[r]^-{\partial} \ar[d]_{h_{\ast}}  & \wt{\Omega}_{8n+2}^{\mathrm{Spin}^{c}}(K(\Q/\Z, 4n-1)) \ar[r]^-{\bar{\psi}_{\ast}} \ar[d]_{\bar{\beta}_{\ast}} & H_{4n+2}(B\mathrm{Spin}^{c}) \ar@{=}[d] \ar[r] & \cdots\\
\cdots \ar[r] & \wt{\Omega}_{r+4n+3}^{\mathrm{Spin}^{c}}(C_{\psi}) \ar[r]^-{\partial}  & \wt{\Omega}_{8n+2}^{\mathrm{Spin}^{c}}(K(\Z, 4n)) \ar[r]^-{\psi_{\ast}}  \ar[d]_-{\varphi}& H_{4n+2}(B\mathrm{Spin}^{c}) \ar[r] & \cdots. \\
 & & \Z/2 && 
}
\end{split}
\end{equation*}
Here, the homomorphism $\psi_{\ast}$ denotes the composition $\mathcal{P} \circ \psi_{\ast} \circ \sigma^{r-4n}\colon$
\[ \wt{\Omega}_{8n+2}^{\mathrm{Spin}^{c}}(K(\Z, 4n)) \to \wt{\Omega}_{r+4n+2}^{\mathrm{Spin}^{c}}( \Sigma^{r-4n} K(\Z, 4n)) \to \wt{\Omega}_{r+4n+2}^{\mathrm{Spin}^{c}}( K(\Z, r) ) \to H_{4n+2}(B\mathrm{Spin}^{c}). \] 
By Lemma \ref{lem:oh}, $\psi_{\ast}$ is given explicitly by 
\begin{equation}\label{eq:psi}
\psi_{\ast}([N,f])=\tau_{N\ast}([N]\cap f^{\ast} (l_{4n})),
\end{equation}
for any bordism class $[N, f]\in  \wt{\Omega}^{\mathrm{Spin}^{c}}_{8n+2}(K(\Z, 4n))$ represented by $f\colon N \rightarrow K(\Z, 4n)$.

\begin{lemma}\label{lem:theta}
%\begin{equation}\label{eq:varphi0}
The composition $\varphi \circ \bar{\beta}_{\ast} \circ \partial = 0$.
%\end{equation}
\end{lemma}

\begin{proof}[Proof of Theorem \ref{thm:theta}]
Since $\Q/\Z$ is a torsion group, $H_{\ast} ( K(\Q/\Z, 4n-1) )$ consists of torsion groups (cf. \cite[p. 77, Lemma 8.8]{gm-b-rh}). 
Consequently, the Atiyah-Hirzebruch spectral sequence implies that $\wt{\Omega}_{8n+2}^{\mathrm{Spin}^{c}}(K(\Q/\Z, 4n-1))$ is also a torsion group.
Therefore, the image of $\bar{\psi}_{\ast}$ must lie in the torsion subgroup of $H_{4n+2}(B\mathrm{Spin}^{c})$.  
Since all torsion in $H_{4n+2}(B\mathrm{Spin}^{c})$ has order $2$ (cf. \cite[p. 317, Corollary]{st68b}), Lemma \ref{lem:theta} implies the existence of a homomorphism $\Theta \colon H_{4n+2}(B\mathrm{Spin}^{c}; \Z) \rightarrow \Z/2$, or equivalently, a cohomology class 
$$\Theta \in \mathrm{Hom}(H_{4n+2}(B\mathrm{Spin}^{c}; \Z), \Z/2) \subset H^{4n+2}(B\mathrm{Spin}^{c}; \Z/2)$$
such that 
\begin{equation}\label{eq:theta}
\Theta \circ \bar{\psi}_{\ast} = \Theta \circ \psi_{\ast} \circ \bar{\beta}_{\ast} = \varphi \circ \bar{\beta}_{\ast}.
\end{equation}
% restricts to $\phi \circ \beta_{\ast}$ on the image of $\wt{\Omega}_{10}^{\mathrm{Spin}^{c}}(K(\Q/\Z, 3))$. 

%For any manifold $N$ and a class $z \in H^{k}(N; G)$, denote by $f_{z} \colon N \rightarrow K(G, k)$ the map satisfies that $f_{z}^{\ast} i_{k} = z$, where $i_{k}$ is the fundamental class of $K(G,k)$.

Now, for any $8n+2$-dimensional spin$^c$ manifold $M$ and any torsion class $t \in TH^{4n}(M)$, the exactness of the sequence \eqref{boc:q/z} implies the existence of an element $z \in H^{4n-1}(M; \Q/\Z)$ such that 
$\beta^{\Q/\Z} (z) = t$.
Therefore, by Identity \eqref{eq:betaf}, we have 
\[ f_{t} = \bar{\beta} \circ f_{z}, \]
and hence $ [M, f_{t}] =  \bar{\beta}_{\ast} ([M, f_{z}]) $.
On the one hand, applying Identity \eqref{eq:theta} yields:
\begin{align*}
\Theta \circ \psi_{\ast} ([M, f_{t}]) & =  \Theta \circ \psi_{\ast} \circ \bar{\beta}_{\ast} ([M, f_{z}]) \\
& = \varphi \circ \bar{\beta}_{\ast} ([M, f_{z}]) \\
& = \varphi ([M, f_{t}]) \\
& = \langle \rho_{2} (t)  \cdot \Sq^{2} \rho_{2} (t) , [M] \rangle.
\end{align*}
On the other hand, by the definition of $\Theta$ and Identity \eqref{eq:psi}, we have
\[ \Theta \circ \psi_{\ast} ([M, f_{t}]) = \Theta ( \tau_{M \ast} ([M] \cap t) ) = \langle \tau_{M}^{\ast}(\Theta), [M]\cap \rho_{2} (t) \rangle = \langle \rho_{2} (t) \cdot \tau_{M}^{\ast} ( \Theta ), [M] \rangle. \]
Comparing these two expressions completes the proof.
\end{proof}

The remainder of this section is devoted to the proof of Lemma \ref{lem:theta}.

Note that $r$ is sufficiently large.
Consider the commutative diagram \eqref{eq:cof} of the cofibrations \eqref{cof:psi} and \eqref{cof:psib}, which induces an exact ladder of cohomology groups for any coefficient group $G$:
\begin{equation}\label{eq:hpsi}
\begin{split}
\xymatrix{
\cdots \ar[r]^-{\delta} & H^{\ast}(C_{\psi}; G) \ar[r]^-{\pi_{\psi}^{\ast}} \ar[d]^{h^{\ast}} & H^{\ast}(K(\Z, r); G) \ar[r]^-{\psi^{\ast}} \ar@{=}[d] & H^{\ast}(\Sigma^{r-4n}K(\Z, 4n); G) \ar[r]^-{\delta} \ar[d]_{\Sigma^{r-4n} \bar{\beta}^{\ast}} & \cdots \\
\cdots \ar[r]^-{\delta} & H^{\ast}(C_{\bar{\psi}}; G) \ar[r]^-{\pi_{\bar{\psi}}^{\ast}} & H^{\ast}(K(\Z, r); G) \ar[r]^-{\bar{\psi}^{\ast}} & H^{\ast}(\Sigma^{r-4n}K(\Q/\Z, 4n-1); G) \ar[r]^-{\delta} & \cdots
}
\end{split}
\end{equation}
(The top and bottom rows are the long exact sequences of $C_{\psi}$ and $C_{\bar{\psi}}$, respectively.)

By analyzing the behavior of the homomorphism $\psi^{\ast}$ with $G=\Z/2$ (cf. Landweber and Stong \cite[pp. 627-628]{ls87}), one finds that 
\begin{enumerate}
\item[i)] $H^{r+4n+1}(C_{\psi}; \Z/2)  \cong \Z/2$, generated by $\overline{\Sq^{4n+1}   l_{r}  }$,
%\label{eq:hpsi5}\\
\item[ii)] $H^{r+4n+3}(C_{\psi}; \Z/2)  \cong (\Z/2)^2$, generated by  $\overline{\Sq^{4n+3}   l_{r}  } $ and $\delta \circ \sigma^{r-4n} (  l_{4n}  \cdot \Sq^{2}  l_{4n} )$, %\label{eq:hpsi7}
%\end{align}
where $\overline{x} \in H^{\ast}(C_{\psi}; \Z/2)$ denotes a class such that $\pi_{\psi}^{\ast}(\overline{x}) = x  \in H^{\ast}(K(\Z, r); \Z/2)$. 
For convenience, the generator of $H^{k}(K(\Z, k); \Z/2) \cong \Z/2$ is also denote by $l_k$.
\end{enumerate}
Landweber and Stong \cite[p. 628 Claim]{ls87} proved the following:
\begin{lemma}\label{lem:sq5}
The generators above satisfy
\begin{equation*}
\Sq^{2}  \overline{\Sq^{4n+1} l_{r}} =\delta \circ \sigma^{r-4n} ( l_{4n} \cdot \Sq^{2} l_{4n} ).
\end{equation*}
\end{lemma}
%Denote by $\phi^{\ast} \colon H^{\ast}(X_{r}; G) \rightarrow H^{\ast}(Y_{r}; G)$ the homomorphism induced by $\phi$ with coefficient $G = \Z$ or $\Z/2$.

Furthermore, by analyzing the cohomology groups of $C_{\psi}$ and $C_{\bar{\psi}}$, one obtains:
\begin{lemma}\label{lem:tor}
The cohomology group $H^{r+4n+1}(C_{\bar{\psi}})$ is a torsion group, and there exists a torsion class $t_{\bar{\psi}} \in H^{r+4n+1}(C_{\bar{\psi}})$ such that 
\[ \rho_{2} (t_{\bar{\psi}}) = h^{\ast} \left( \overline{ \Sq^{4n+1} l_{r} } \right). \]
%where $\phi^{\ast} \colon H^{r+5}(X_{r}; \Z/2) \rightarrow H^{r+5}(Y_{r}; \Z/2)$ is the induced homomorphism.% \colon H^{r+5}(X_{r};\Z/2) \rightarrow H^{r+5}(Y_{r};\Z/2)$ is the homomorphism induced by $\phi$.
\end{lemma}

\begin{proof}
Since $r$ is large, we note the following facts:
\begin{enumerate}
\item[$(1)$] $\Q/Z$ is a torsion group implies that 
$$H^{r+4n}(\Sigma^{r-4n}K(\Q/\Z, 4n-1); \Q) \cong H^{8n}( K(\Q/\Z, 4n-1); \Q ) = 0$$
by \cite[p. 77, Lemma 8.8]{gm-b-rh}.
\item [$(2)$] $H^{r+4n}(K(\Z, r); \Q) = H^{r+4n+1}(K(\Z, r); \Q) = 0$ by \cite[p. 550, Proposition 5.21]{hat-bss}.
\item [$(3)$] $H^{r+4n}(\Sigma^{r-4n}K(\Z,4n); \Q) \cong H^{8n}(K(\Z, 4n); \Q) \cong \Q$ by \cite[p. 550, Proposition 5.21]{hat-bss}.
\end{enumerate}

Facts $(1)$ and $(2)$, combined with the bottom row of the exact ladder \eqref{eq:hpsi} for $G = \Q$, imply that $H^{r+4n+1}(C_{\bar{\psi}}; \Q) =0$.
Hence, $H^{r+4n+1}(C_{\bar{\psi}})$ is a torsion group.

To prove the existence of $t_{\bar{\psi}}$, consider the cohomology group $H^{r+4n+1}(C_{\psi})$.
By construction and the Freudenthal suspension theorem (Lemma \ref{lem:freu}), $C_{\psi}$ is $(r{+}4n)$-connected. 
The universal coefficient theorem then implies that $H^{r+4n+1}(C_{\psi})$ is torsion free. 
Moreover,  combining Facts $(2)$ and $(3)$ with the top row of the ladder \eqref{eq:hpsi} for $G=\Q$, we find $H^{r+4n+1}(C_{\psi}; \Q) \cong \Q$.
Therefore, 
\[ H^{r+4n+1}(C_{\psi}) \cong \Z. \]
The Bockstein sequence \eqref{eq:bseq} now implies the existence of a class $x \in H^{r+4n+1}(C_{\psi})$ such that $\rho_{2} (x) = \overline{ \Sq^{4n+1} l_{r} }$. 
Set 
$$t_{\bar{\psi}} = h^{\ast} (x) \in H^{r+4n+1}(C_{\bar{\psi}}).$$
Then $t_{\bar{\psi}}$ is a torsion class and 
\begin{equation*}
\rho_{2} (t_{\bar{\psi}}) = \rho_{2} ( h^{\ast} (x) ) = h^{\ast} ( \rho_{2} (x) ) = h^{\ast}  \left( \overline{ \Sq^{4n+1} l_{r} } \right),
\end{equation*}
which completes the proof.
\end{proof}

\begin{proof}[Proof of Lemma \ref{lem:theta}]
Consider any bordism class 
$$[(W,\partial W), (f, g)] \in \wt{\Omega}_{r+4n+3}^{\mathrm{Spin}^{c}}(C_{\bar{\psi}}) \cong \Omega_{r+4n+3}^{\mathrm{Spin}^{c}}(K(\Z, r), \Sigma^{r-4n}K(\Q/\Z, 4n-1))$$
represented by maps $f,~g$ fitting into the commutative diagram:
\begin{equation*}
\begin{split}
\xymatrix{
\partial W \ar[r]^-{g} \ar@{_{(}->}[d]^{} & \Sigma^{r-4n}K(\Q/\Z, 4n-1) \ar[d]^{\bar{\psi}} \\
W \ar[r]^-{f} & K(\Z, r).
}
\end{split}
\end{equation*}
From the definition of $\varphi$ and Lemmas \ref{lem:sq5} and \ref{lem:tor}, we compute
\begin{align*}
\varphi \circ \bar{\beta}_{\ast} \circ \partial ([(W,\partial W), (f, g)]) & = \langle g^{\ast} \circ (\Sigma^{r-4n}\bar{\beta})^{\ast} \circ \sigma^{r-4n} (l_{4n} \cdot \Sq^{2} l_{4n}), [\partial W] \rangle \\
&= \langle \delta \circ g^{\ast}  \circ (\Sigma^{r-4n}\bar{\beta})^{\ast}  \circ \sigma^{r-4n}  ( l_{4n} \cdot \Sq^{2} l_{4n} ), [W, \partial W] \rangle \\
&= \langle f^{\ast} \circ h^{\ast} \circ \delta \circ \sigma^{r-4n} ( l_{4n} \cdot \Sq^{2} l_{4n}), [W,\partial W] \rangle \\
&= \langle f^{\ast} \circ h^{\ast} \circ \Sq^{2}  \overline{ \Sq^{4n+1} l_{r} } , [W,\partial W] \rangle \\
&= \langle f^{\ast} \circ \Sq^{2} \rho_{2} (t_{\bar{\psi}}) , [W, \partial W] \rangle.
\end{align*}
The Wu class $v_2(W)$ is defined as in \cite[equation (7.1)]{ker-rcc}.
Since $W$ is orientable, Wu's formula \eqref{eq:wu} together with \cite[Lemma (7.3)]{ker-rcc} implies that $v_2(W) = w_2(W)$.
Therefore, by the definition of Wu class, we have 
\begin{align*}
\varphi \circ \bar{\beta}_{\ast} \circ \partial ([(W,\partial W), (f, g)])
&= \langle f^{\ast} \circ \Sq^{2} \rho_{2} (t_{\bar{\psi}}) , [W, \partial W] \rangle \\
& = \langle  w_{2}(W) \cdot  f^{\ast} (\rho_{2} (t_{\bar{\psi}}) ), [W, \partial W] \rangle.
\end{align*}
Since $W$ is spin$^{c}$, there exists an element $c \in H^{2}(W)$ such that $\rho_{2} (c) = w_{2}(W)$. 
By Lemma \ref{lem:tor}, $t_{\bar{\psi}}$ is a torsion element. 
Therefore, $c \cdot f^{\ast} (t_{\bar{\psi}})$ is a torsion element in $H^{r+4n+3}(W, \partial W) \cong \Z$, hence must be zero. 
Consequently,
\begin{align*}
\varphi \circ \bar{\beta}_{\ast} \circ \partial ([(W,\partial W), (f, g)]) & = \langle  w_{2}(W)\cdot f^{\ast} ( \rho_{2} (t_{\bar{\psi}}) ), [W, \partial W] \rangle \\
&=\langle \rho_{2}(c \cdot f^{\ast} (t_{\bar{\psi}}) ), [W, \partial W] \rangle \\
&=0.
\end{align*}
This completes the proof.
\end{proof}

% {s:exis} 
%%%%%%%%%%%%%%%%%%%%%%%%%%%%%%%%%%%%%%%%%%%%%%%

\section{Describing $\Theta$}
\label{s:prop}
%%%%%%%%%%%%%%%%%%%%%%%%%%%%%%%%%%%%%%%%%%%%%%%%%%%%%%
This section establishes some elementary properties of the class $\Theta \in H^{4n+2}(B\mathrm{Spin}^c; \Z/2)$ whose existence is guaranteed by Theorem \ref{thm:theta}.
\begin{proposition}\label{prop:thetaz}
The class $\Theta$ is well-defined only modulo the subgroup $\rho_{2} (H^{4n+2}(B\mathrm{Spin}^{c}))$.
That is, it is uniquely determined as an element of the quotient group
\begin{equation*}
H^{4n+2}(B\mathrm{Spin}^{c};\Z/2)/\rho_{2} (H^{4n+2}(B\mathrm{Spin}^{c})).
\end{equation*}
\end{proposition}

\begin{proof}
Let $M$ be an $(8n{+}2)$-dimensional spin$^{c}$ manifold $M$. 
For any class $x \in H^{4n+2}(B\mathrm{Spin}^{c})$ and any torsion element $t \in TH^{4n}(M)$, the cup product $\tau_M^{\ast} (x) \cdot t$ is a torsion class in $H^{8n+2}(M) \cong \Z$. Consequently, $\tau_M^{\ast} (x) \cdot t = 0$.
We then compute
\begin{align*}
\tau_M^{\ast}(\Theta+\rho_{2} (x) ) \cdot \rho_{2} (t) &= \tau_M^{\ast}(\Theta) \cdot \rho_{2} (t) + \rho_{2} (\tau_M^{\ast} (x) ) \cdot \rho_{2} (t) \\
&=  \rho_{2} (t) \cdot \Sq^{2} \rho_{2} (t) + \rho_{2}(\tau_M^{\ast} (x) \cdot t)\\
&= \rho_{2} (t) \cdot \Sq^{2} \rho_{2} (t).
\end{align*}
Thus, the class $\Theta+\rho_{2} (x)$ satisfies the same defining property as $\Theta$, which completes the proof.
\end{proof}

\begin{proposition}\label{prop:theta0}
The class $\Theta$ is nonzero in $H^{4n+2}(B\mathrm{Spin}^{c};\Z/2)/\rho_{2} ( H^{4n+2}(B\mathrm{Spin}^{c}) )$.
Consequently, both $\beta^{\Z/2} ( \Theta ) \in H^{4n+3}(B\mathrm{Spin}^{c}) $ and $Sq^1 \Theta \in H^{4n+3}(B\mathrm{Spin}^{c}; \Z/2)$ are nonzero.
Furthermore, the class 
\[ \Theta \in H^{4n+2}(B\mathrm{Spin}^{c};\Z/2)/\rho_{2} ( H^{4n+2}(B\mathrm{Spin}^{c}) )\]
is uniquely determined by $Sq^1 \Theta$.
\end{proposition}

\begin{proof}
Consider the homomorphism $\varphi \colon \wt{\Omega}_{8n+2}^{\mathrm{Spin}^c}(K(\Z, 4n)) \to \Z/2$ defined by 
$$\varphi([N, f])=\langle f^{\ast}(l_{4n})\cdot \Sq^{2} f^{\ast}(l_{4n}), ~[N] \rangle.$$
Now examine the following commutative diagram:
\begin{equation*}
\begin{split}
\xymatrix{
 \wt{\Omega}_{8n+2}^{\mathrm{Spin}}(K(\Q/\Z, 4n{-}1)) \ar[r]^-{\bar{\beta}_{\ast}} \ar[d]_{i} &\wt{\Omega}_{8n+2}^{\mathrm{Spin}}(K(\Z, 4n))   \ar[d]_{i}  &   \\
  \wt{\Omega}_{8n+2}^{\mathrm{Spin}^{c}}(K(\Q/\Z, 4n{-}1)) \ar[r]^-{\bar{\beta}_{\ast}}  & \wt{\Omega}_{8n+2}^{\mathrm{Spin}^{c}}(K(\Z, 4n))  \ar[r]^-{\varphi} & \Z/2.
}
\end{split}
\end{equation*}
Here, the vertical maps $i$ are the natural forgetful homomorphisms from spin to spin$^c$ bordism.
By Lemma \ref{lem:qz}, the homomorphism $\bar{\beta}_{\ast}$ on the top row is an isomorphism.
Furthermore, according to Landweber and Stong \cite[lemma 3.2]{ls87}, the composition
$ \varphi \circ i $ on the top right is nonzero.
It follows that the composition $\varphi \circ i \circ \bar{\beta}_{\ast}$ on the top left is nontrivial.
By commutativity of the diagram, the composition $\varphi \circ \bar{\beta}_{\ast}$ on the bottom row must also be nonzero.
Theorem \ref{thm:theta} and Proposition \ref{prop:thetaz} then imply that $\Theta \neq 0$.
%\colon \wt{\Omega}_{8n+2}^{\mathrm{Spin}}(K(\Z, 4n)) \to \Z/2 \]

The remaining assertions follow from the Bockstein sequence \eqref{eq:bseq} for $X = B\mathrm{Spin}^c$ and the fact that all torsion in $H^{\ast}(B\mathrm{Spin}^c)$ has order $2$ (cf. \cite[p. 317, Corollary]{st68b}).
% are the homomorphisms induced from the Bockstein map $\bar{\beta} \colon K(\Q/\Z, 4n{-}1) \rightarrow K(\Z, 4n)$ as \eqref{map:beta}.
%Since the composition homomorphism 
%\[ \varphi \circ i \colon \wt{\Omega}_{8n+2}^{\mathrm{Spin}}(K(\Z, 4n)) \to \Z/2 \]
%is nonzero by Landweber and Stong \cite[lemma 3.2]{ls87} and 
%\[ \bar{\beta}_{\ast} \colon \wt{\Omega}_{8n+2}^{\mathrm{Spin}}(K(\Q/\Z, 4n{-}1)) \to \wt{\Omega}_{8n+2}^{\mathrm{Spin}}(K(\Z, 4n)) \]
%is an isomorphism by Lemma \ref{lem:qz}, we conclude that 
%\[ \varphi \circ i \circ \bar{\beta}_{\ast} \colon \wt{\Omega}_{8n+2}^{\mathrm{Spin}}(K(\Q/\Z, 4n-1)) \to \Z/2 \]
%is nontrivial. Therefore, the commutative diagram above tell us that 
%\[ \varphi \circ \bar{\beta}_{\ast} \colon \wt{\Omega}_{8n+2}^{\mathrm{Spin}^c}(K(\Q/\Z, 4n-1)) \to \Z/2 \]
%is nonzero, and hence $\Theta \neq 0$ by Theorem \ref{thm:theta} and Proposition \ref{prop:thetaz}.
%
%The remaining facts of this proposition can be deduced easily from the Bockstein sequence \eqref{eq:bseq} for $X = B\mathrm{Spin}^c$ and the fact that all torsion in $H^{\ast}(B\mathrm{Spin}^c)$ has order $2$ by \cite[p. 317, Corollary]{st68b}.
\end{proof}

%Proposition \ref{prop:theta0} is based on the results of Landweber and Stong \cite{ls87} and the following lemma. 

\begin{proposition}
\label{prop:sqtheta}
For any $(8n{+}1)$-dimensional spin$^c$ manifold $M$, we have 
\[  \beta^{\Z/2} ( \tau_M^{\ast} ( \Theta ) ) = 0 \]
and hence $ Sq^1 \tau_M^{\ast}( \Theta  ) = 0$.
\end{proposition}
\begin{remark} \label{rem:zero}
This result implies that  $\beta^{\Z/2} ( \tau_M^{\ast} ( \Theta ) ) = 0$
and $ Sq^1 \tau_M^{\ast}( \Theta  ) = 0$ for any spin$^c$ manifold $M$ of dimension less than or equal to $8n{+}1$.
\end{remark}

The proof of Proposition \ref{prop:sqtheta} relies on the following lemma.
\begin{lemma}\label{lem:rho2}
Let $M$ be an $m$-dimensional manifold. 
For any $x \in H^k(M; \Z/2)$, the following three statements are equivalent:
\begin{enumerate}
\item[(1)] $ \bocj{} ( x )= 0$;
\item[(2)] There exists an integral class $z \in H^{k}(M)$ such that $\rho_{2}(z) = x$;
\item[(3)] $t \cdot x = 0$ for any torsion class $t \in TH^{m-k}(M)$.
\end{enumerate}
\end{lemma}

\begin{proof}
The Poincar\'e Duality Theorem implies that the bilinear form 
\[\cup \colon H^{k}(M; \Z/2) \times H^{m-k}(M;\Z/2) \to H^m(M;\Z/2) \cong \Z/2\]
is nondegenerate.
By Massey \cite[Lemma 1]{ma62}, The image $\rho_{2} ( H^{k}(M) )$ is the annihilator of $\rho_{2}( TH^{m-k}(M) )$.
The claimed equivalences now follow from this fact combined with the exactness of the Bockstein sequence \eqref{eq:bseq}.
\end{proof}

\begin{proof}[Proof of Proposition \ref{prop:sqtheta}]

Define homomorphisms
$$ \begin{array}{ll} \vspace{8pt}
\cdot \Theta \colon \wt{\Omega}_{8n+1}^{\mathrm{Spin}^{c}}(K(\Z, 4n-1)) \rightarrow \Z/2, &
\qquad [N, f] \mapsto \langle f^{\ast}(l_{4n-1}) \cdot \tau_N^{\ast}(\Theta), [N] \rangle,\\ \vspace{8pt}
\cdot \Theta \colon \wt{\Omega}_{8n+2}^{\mathrm{Spin}^{c}}(\Sigma K(\Z, 4n-1)) \rightarrow \Z/2, &
\qquad [N, f] \mapsto \langle f^{\ast} ( \sigma (l_{4n-1}) ) \cdot \tau_N^{\ast}(\Theta), [N] \rangle,\\
\cdot \Theta \colon \wt{\Omega}_{8n+2}^{\mathrm{Spin}^{c}}(K(\Z, 4n)) \rightarrow \Z/2, &
\qquad [N, f] \mapsto \langle f^{\ast} (l_{4n}) \cdot \tau_N^{\ast}(\Theta), [N] \rangle.
\end{array} $$
These fit into a commutative diagram:
\begin{equation*}
\begin{split}
\xymatrix{
& \wt{\Omega}_{8n+1}^{\mathrm{Spin}^{c}}(K(\Q/\Z, 4n-2)) \ar[r]^-{\sigma} \ar[d]_{\bar{\beta}_{\ast}} & \wt{\Omega}_{8n+2}^{\mathrm{Spin}^{c}}(\Sigma K(\Q/\Z, 4n-2))  \ar[d]_{\Sigma\bar{\beta}_{\ast}} & \\%\Omega_{10}^{Spin^{c}}(K(\Q/\Z, 3)) \ar[d]_{\beta_{\ast}}\\
& \wt{\Omega}_{8n+1}^{\mathrm{Spin}^{c}}(K(\Z, 4n-1)) \ar[r]^-{\sigma} \ar[d]_{\cdot \Theta} & \wt{\Omega}_{8n+2}^{\mathrm{Spin}^{c}}(\Sigma K(\Z, 4n-1)) \ar[r]^-{\psi_{\ast}} \ar[d]_{\cdot \Theta} & \wt{\Omega}_{8n+2}^{\mathrm{Spin}^{c}}(K(\Z, 4n)) \ar[d]_{\cdot \Theta}\\
& \Z/2 \ar@{=}[r] & \Z/2 \ar@{=}[r] & \Z/2,
}
\end{split}
\end{equation*}
where $\bar{\beta}_{\ast}$, $\Sigma \bar{\beta}_{\ast}$ and $\psi_{\ast}$ are the homomorphisms induced from $\bar{\beta},~\Sigma \bar{\beta}$ and $\psi$, respectively, and $\psi \colon \Sigma K(\Z, 4n-1) \to K(\Z, 4n)$ is the map satisfying $\psi^{\ast} ( l_{4n}) = \sigma (l_{4n-1})$.

We claim that the composition $\cdot \Theta \circ \psi_{\ast} \circ \Sigma \bar{\beta}_{\ast}$ is the zero map; the proof is given below.
This implies
\begin{equation}\label{eq:w6beta}
\cdot \Theta \circ \bar{\beta}_{\ast}  = 0,
\end{equation}
by the commutativity of the diagram.
 
Now, for any torsion class $t \in TH^{4n-1}(M)$, there exists an element $z \in H^{4n-2}(M; \Q/\Z)$
such that $\bocqz (z) = t$.
By Equation \eqref{eq:betaf}, we have $\bar{\beta} \circ f_{z} = f_{t}$.
Applying Equation \eqref{eq:w6beta} yields
\begin{align*}
\langle t \cdot \tau_M^{\ast}(\Theta), [M] \rangle =  \cdot \Theta ([M, f_{t}]) = \cdot \Theta ([M, \bar{\beta} \circ f_{z}]) =  \cdot \Theta \circ \bar{\beta}_{\ast} ([M, f_{z}]) = 0.
\end{align*}
Since this holds for all torsion classes $t \in TH^{4n-1}(M)$, Lemma \ref{lem:rho2} implies that
$\beta^{\Z/2} ( \tau_M^{\ast} ( \Theta ) ) = 0$, and hence $ Sq^1 \tau_M^{\ast}( \Theta  ) = 0$, which completes the proof.

It remains to prove the claim.
Set 
\[ t := \Sigma \bar{\beta}^{\ast} \circ \psi^{\ast} (l_{4n}) \in H^{4n}(\Sigma K(\Q/\Z, 4n{-}2)). \]
For any $[N, f] \in \wt{\Omega}_{8n+2}^{\mathrm{Spin}^{c}}(\Sigma K(\Q/\Z, 4n-2))$, since $t$ is a torsion class and the cup product on
$\wt H^{*}(\Sigma K(\Q/\Z, 4n-2))$ is trivial, Theorem \ref{thm:theta} implies that  
\begin{equation*}
\cdot \Theta \circ \psi_{\ast} \circ \Sigma \bar{\beta}_{\ast} ([N, f]) = \langle f^{\ast} (t) \cdot \tau_N^{\ast}(\Theta), [N] \rangle = \langle f^{\ast} ( t \cdot \Sq^{2} ( t ) ), [N] \rangle = 0,
\end{equation*}
which completes the proof of the claim.
\end{proof}
%\begin{align*}
% \cdot \Theta \circ \psi_{\ast} \circ \Sigma \bar{\beta}_{\ast}  ([N, f]) = \cdot \Theta ([N, \psi \circ \Sigma \bar{\beta} \circ f]) = \langle  f^{\ast} \circ \Sigma \bar{\beta}^{\ast} \circ \psi^{\ast} (l_{4n}) \cdot \tau_N^{\ast}(\Theta), [N] \rangle.
%\end{align*}

%
%Obviously $t$ is a torsion class, hence so is $f^{\ast} (t)$, i.e., $f^{\ast} (t) \in TH^{4n}(N)$.
% is a torsion class. 

%where \[ t = \Sigma \bar{\beta}^{\ast} \circ \psi^{\ast} (l_{4n}) \in H^{4n}(\Sigma K(\Q/\Z, 4n{-}2)) \]
%is a torsion class.
%This means that the homomorphism $\cdot \Theta \circ \psi_{\ast} \circ \Sigma \bar{\beta}_{\ast}$ is trivial.

%{s:des}
%%%%%%%%%%%%%%%%%%%%%%%%%%%%%%%%%%%%%%%%%%%%%%%%%%%%%%

\section{Proof of Theorem \ref{thm:main}}
\label{s:main}
%%%%%%%%%%%%%%%%%%%%%%%%%%%%%%%%%%%%%%%%%%%%%%%%%%%%%%%
Building on the results from Sections \ref{s:exis} and \ref{s:prop}, this section is devoted to the proof of Theorem \ref{thm:main}.
For convenience, throughout this section, $H^{\ast}(X)$ will denote the mod $2$ cohomology ring of a $CW$-complex $X$.

\subsection{Outline of the Proof}
\label{ss:idea}
According to Theorem \ref{thm:theta}, proving Theorem \ref{thm:main} reduces to determining the class $\Theta \in H^{4n+2}(B\mathrm{Spin}^c)$.
By Proposition \ref{prop:theta0}, this is equivalent to identifying the class $Sq^1 \Theta \in H^{4n+3}(B\mathrm{Spin}^c)$.
The identification of this class is guided by Propositions \ref{prop:theta0} and \ref{prop:sqtheta}.

Proposition \ref{prop:sqtheta} implies that $Sq^1 \Theta  \neq 0 \in H^{4n+3}(B\mathrm{Spin}^c)$.
By the universal coefficient theorem, this means:
\begin{lemma}\label{lem:nozero}
There exists an element $x \in H_{4n+3}(B\mathrm{Spin}^c)$ such that $\langle Sq^1 \Theta, x \rangle \neq 0$. \qed
\end{lemma}

Furthermore, Proposition \ref{prop:sqtheta} and Remark \ref{rem:zero} imply that $Sq^1 \tau_{M}^{\ast} (\Theta) = 0 \in H^{4n+3}(M)$ for any $(8n{-}1)$-dimensional spin$^c$ manifold $M$.
By the Poincar\'e Duality Theorem, this implies:
\begin{lemma}\label{lem:zero}
For any $(8n{-}1)$-dimensional spin$^c$ manifold $M$ and any class $y \in H^{4n-4}(M)$, we have
\[ \langle y \cdot Sq^1 \tau_{M}^{\ast} ( \Theta ), [M] \rangle = \langle Sq^1 \Theta, \tau_{M\ast}([M] \cap y) \rangle = 0, \]
where $\tau_M \colon M \rightarrow B\mathrm{Spin}^{c}$ classifies the stable tangent bundle of $M$. \qed
\end{lemma}

Analogous to the definition of $\mathcal{P}$ in Section \ref{s:exis}, for any positive integers $i$ and $r$, define a homomorphism
\[ \mathcal{P}_2 \colon \wt{\Omega}_{r+i}^{\mathrm{Spin}^c}( K(\Z/2, r) ) \to H_i(B\mathrm{Spin}^c) \]
by
\[ \mathcal{P}_2 ([N, f]) = \tau_{N\ast} ( [N] \cap f^{\ast}(l_r) ) \]
for any bordism class $[N, f]\in  \wt{\Omega}^{\mathrm{Spin}^c}_{r+i}(K(\Z/2, r))$ represented by $f\colon N \rightarrow K(\Z/2, r)$.
Here and subsequently, the generator of $H^{k}(K(\Z/2, k)) \cong \Z/2$ is also denoted by $l_k$.
We have the following lemma.
\begin{lemma}\label{lem:oh2}
For any fixed positive integer $i$ and sufficiently large $r$, the map
\begin{align*}
\mathcal{P}_2 \colon \wt{\Omega}_{i+r}^{\mathrm{Spin}^c}( K(\Z/2, r) ) \to H_i(B\mathrm{Spin}^c) 
\end{align*}
is an isomorphism. \qed
\end{lemma}

For any positive integer $m$ and large $r$, consider the cofibration
\begin{align} \label{cof:Psi}
\Sigma^{r-4m}K(\Z/2, 4m)  \xrightarrow{\varPsi} K(\Z/2,r) \xrightarrow{\pi_{\varPsi}} C_{\varPsi},
\end{align}
where $\varPsi \colon \Sigma^{r-4m}K(\Z/2, 4m) \rightarrow K(\Z/2, r)$
is the map satisfying
$\varPsi^{\ast} l_{r} = \sigma^{r-4m} l_{4m}$.
This cofibration induces the following diagram:
\begin{equation}\label{seq:Psi}
\begin{split}
\xymatrix{
\wt{\Omega}^{\mathrm{Spin}^c}_{r+4m+7}( \Sigma^{r-4m}K(\Z/2, 4m) ) \ar[r]^-{\varPsi_{\ast}} & \wt{\Omega}^{\mathrm{Spin}^c}_{r+4m+7}( K(\Z/2, r) ) \ar[r]^-{\pi_{\varPsi \ast}}  \ar[d]_-{\mathcal{P}_2} ^-{\cong}& \wt{\Omega}^{\mathrm{Spin}^c}_{r+4m+7}( C_{\varPsi} ) \\
\wt{\Omega}^{\mathrm{Spin}^c}_{8m+7}( K(\Z/2, 4m) ) \ar[u]^{\sigma^{r-4m}}_-{\cong} & H_{4m+7}(B\mathrm{Spin}^c) & 
}
\end{split}
\end{equation}
where the horizontal sequence is the exact sequence of reduced bordism groups induced by the cofibration \eqref{cof:Psi}, 
$\sigma^{r-4m}$ is the $(r{-}4m)$-fold suspension isomorphism, 
and $\mathcal{P}_2$ is the isomorphism defined above. 
%Here and subsequently, for convenience, the generator of $H^{k}(K(\Z/2, k)) \cong \Z/2$ is also denote by $l_k$.
%
%Then it follows from Lemma \ref{lem:oh2} and the suspension isomorphism \eqref{iso:sus} that we have the following exact sequence of the cofibration \eqref{cof:Psi}:
%\begin{equation}\label{seq:Psi}
%\cdots \to \wt{\Omega}^{\mathrm{Spin}^c}_{8m+7}( K(\Z/2, 4m)) \xrightarrow{\varPsi_{\ast}} H_{4m+7}(B\mathrm{Spin}^c) \xrightarrow{\pi_{\varPsi\ast}} \wt{\Omega}^{\mathrm{Spin}^c}_{r+4m+7}(C_{\varPsi}) \to \cdots. 
%\end{equation}
%Here we denote the composition map $\mathcal{P}_2 \circ \varPsi_{\ast} \circ \sigma^{r-4m} \colon$
%\[ \wt{\Omega}_{8m+7}^{\mathrm{Spin}^{c}}(K(\Z/2, 4m)) \to \wt{\Omega}_{r+4m+7}^{\mathrm{Spin}^{c}}( \Sigma^{r-4m} K(\Z/2, 4m)) \to \wt{\Omega}_{r+4m+7}^{\mathrm{Spin}^{c}}( K(\Z/2, r) ) \to H_{4m+7}(B\mathrm{Spin}^{c}) \]
%briefly by $\varPsi_{\ast}$.  
It follows easily from Lemma \ref{lem:oh2} that the composition $\mathcal{P}_2 \circ \varPsi_{\ast} \circ \sigma^{r-4m}$ is given by 
\begin{equation}\label{eq:Psi}
\mathcal{P}_2 \circ \varPsi_{\ast} \circ \sigma^{r-4m} ([N,f]) = \tau_{N\ast}([N]\cap f^{\ast} (l_{4m})),
\end{equation}
for any bordism class $[N, f]\in  \wt{\Omega}^{\mathrm{Spin}^{c}}_{8m+7}(K(\Z/2, 4m))$ represented by $f\colon N \rightarrow K(\Z/2, 4m)$.

Now, set $m = n{-}1$.
For any bordism class $[N, f]\in  \wt{\Omega}^{\mathrm{Spin}^{c}}_{8n-1}(K(\Z/2, 4n{-}4))$, Lemma \ref{lem:zero} and Equation \eqref{eq:Psi} imply that 
\begin{align*}
\langle Sq^1 \Theta, \mathcal{P}_2 \circ \varPsi_{\ast} \circ \sigma^{r-4n+4} ([N,f])\rangle = \langle Sq^1 \Theta, \tau_{N\ast}([N]\cap f^{\ast} (l_{4n-4})) \rangle  = 0.
\end{align*}
This means that for any $x \in \mathrm{Im}  ( \mathcal{P}_2 \circ \varPsi_{\ast} ) \subset H_{4n+3}(B\mathrm{Spin}^c)$, we must have
\begin{equation*}
\langle Sq^1 \Theta, x \rangle = 0.
\end{equation*}
Since $\mathcal{P}_2$ is an isomorphism,
combining this fact with Lemma \ref{lem:nozero} 
shows that $\varPsi_{\ast}$ is not surjective.
Therefore, in some sense, $Sq^1 \Theta$ must lie in the cokernel of $\varPsi_{\ast}$, i.e., the image of $\pi_{\varPsi\ast}$.
Thus, to determine $Sq^1 \Theta$, it is necessary to compute the spin$^c$ bordism group $\wt{\Omega}^{\mathrm{Spin}^c}_{r+4m+7}(C_{\varPsi}) $, 
and identify the image of $\pi_{\varPsi\ast}$.

The computation of $\wt{\Omega}^{\mathrm{Spin}^c}_{r+4m+7}(C_{\varPsi}) $ is lengthy and constitutes the majority of this section.
Recall that $M\mathrm{Spin}^{c}(8s)$ is the Thom space of the classifying bundle over $B\mathrm{Spin}^{c}(8s)$.
For large $s$, we have the isomorphism
\[ \Omega_{r+4m+7}^{\mathrm{Spin}^c} (C_{\varPsi}) \cong \pi_{r+8s+4m+7}(M\mathrm{Spin}^c(8s) \wedge C_{\varPsi}). \]
For convenience, let $\mathcal{M}$ denote the smash product $M\mathrm{Spin}^c(8s) \wedge C_{\varPsi}$.
The strategy for computing this bordism group is as follows:
First, determine the mod $2$ cohomology groups of $\mathcal{M}$;
then, select a set of generators to construct a map $f$ from $\mathcal{M}$ to a product of Eilenberg-MacLane spaces;
finally, prove that $f$ induces an isomorphism on the  $(r{+}8s{+}4m{+}7)$-th homotopy groups, thereby fully determining the bordism group $\Omega_{r+4m+7}^{\mathrm{Spin}^c} (C_{\varPsi}) $.

This proof strategy is due to Landweber and Stong \cite{ls87}. 

The remainder of Section \ref{s:main} is organized as follows.
After some preliminaries in Subsection \ref{ss:pre}, the mod $2$ cohomology groups of $C_{\varPsi}$ and $M\mathrm{Spin}^c(8s)$ are described in Subsections \ref{ss:cofi} and \ref{ss:mspin}, respectively. 
The bordism group $\Omega_{r+4m+7}^{\mathrm{Spin}^c} (C_{\varPsi})$ is determined in Subsection \ref{ss:borCPsi}, and the class $Sq^1 \Theta$ is identified in Subsection \ref{ss:pfmain}.

%and hence the $\bmod ~2$ cohomology groups of $M\mathrm{Spin}^c(8s) \wedge C_{\varPsi}$ are determined at the beginning of subsection \ref{ss:borCPsi}.
%In subsection \ref{ss:borCPsi}, some generators of the $\bmod ~2$ cohomology groups of $M\mathrm{Spin}^c(8s) \wedge C_{\varPsi}$ are selected to construct a map from $M\mathrm{Spin}^c(8s) \wedge C_{\varPsi}$ to a product of some Eilenberg-MacLane spaces, and we will prove that this map induces an isomorphism between the  $(r+8s+4m+7)$-th homotopy groups, which means 

%in order to find out $Sq^1 \Theta$, it is helpful to determine the spin$^c$ cobordism group $\wt{\Omega}^{\mathrm{Spin}^c}_{r+4m+7}(C_{\varPsi}) $ of $C_{\varPsi}$.

%In summary, our idea of the proof of Theorem \ref{thm:main} is: first step, computing the spin$^c$ cobordism group $\wt{\Omega}^{\mathrm{Spin}^c}_{r+4m+7}(C_{\varPsi}) $; 

\subsection{Preliminaries}
\label{ss:pre}

To compute the spin$^c$ bordism group $\wt{\Omega}^{\mathrm{Spin}^c}_{r+4m+7}(C_{\varPsi})$ and prove Theorem \ref{thm:main}, we require some preliminaries.
%For any $CW$-complex $X$, denote by $\Sigma X$ the suspension of $X$. 
%Denote by $\wt{\Omega}_{i}^{\mathrm{Spin}} (X)$ (resp. $\wt{\Omega}_{i}^{\mathrm{Spin}^c} (X)$) the $i$-th reduced spin (resp. spin$^c$) cobordism group of $X$.
%It is known that 
%$$\wt{\Omega}_{i}^{G} (X) \cong \pi_{i}(MG \wedge X) = \varinjlim_{s} \pi_{i+s}(MG(8s) \wedge X).$$

%Take $X = K(\Z, r)$.
%For large $r$, we have 
%\begin{align}
%\wt{\Omega}_{i+r}^{Spin^c} (K(\Z, r)) \cong \pi_{i+r}(MSpin^c \wedge K(\Z, r)) = \varinjlim_{s} \pi_{i+s+r}(MG(8s) \wedge X).
%\end{align}

For any $CW$-complex $X$, denote by
\[  Sq^i \colon H^k(X; \Z/2) \to H^{k+i}(X; \Z/2), ~i \ge 0 \]
the Steenrod squares.
%It is known that the satisfying the following list of properties.
%\begin{enumerate}[itemsep=0pt, topsep=0pt, parsep=5pt]
%\item[(1)] $Sq^i f^{\ast} (\alpha) = f^{\ast} ( Sq^i \alpha )$ for $f \colon X \to Y$.
%\item[(2)] $Sq^i ( \alpha + \beta ) = Sq^i \alpha + Sq^i \beta$.
%\item[(3)] $Sq^0 = id$, the identity.
%\item[(4)] $Sq^1 = \rho_2 \circ \beta^{\Z/2}$.
%\item[(5)] $Sq^i \alpha = \alpha^2$ if $| \alpha | = i$, and $Sq^i \alpha = 0$ if $| \alpha | < i$.
%\item[(6)] $Sq^i(\alpha \beta) = \sum_j Sq^j \alpha Sq^{i-j} \beta$ (The Cartan formula).
%\item[(7)] $Sq^i \sigma \alpha = \sigma Sq^i \alpha$ where $\sigma \colon H^k(X) \to H^{k+1}(\Sigma X)$ is the suspension isomorphism given by reduced cross product with a generator of $H^1(S^1)$.
%\end{enumerate}
These are homomorphisms satisfying naturality; $Sq^0$ is the identity map; $Sq^1 = \rho_2 \circ \beta^{\Z/2}$ (see sequence \eqref{eq:bseq}); $Sq^i x = x^2$ if $|x| = i$, and $Sq^i x = 0$ if $|x| < i$.
Moreover, the Steenrod squares commute with the suspension isomorphism $\sigma$, i.e., $Sq^i \circ \sigma = \sigma \circ Sq^i$, $i \ge 0$,
and satisfy the Cartan formula:
\begin{equation}\label{eq:cartan}
Sq^i (x \cdot y) = \Sigma_{j} Sq^j x \cdot Sq^{i-j} y.
\end{equation}
Compositions of Steenrod squares satisfy the Adem relations:
\begin{align}\label{eq:adem}
Sq^a Sq^b = \sum_{c=0}^{[a/2]}\binom{b-1-c}{a-2c}Sq^{a+b-c} Sq^c
\end{align}
where $0 < a < 2b$, and $[a/2]$ denotes the greatest integer less than or equal to $a/2$. 
By convention, the binomial coefficient $\binom{x}{y}$ is zero if $x$ or $y$ is negative, or if $x < y$;
also, $\binom{x}{0} = 1$ for $x \ge 0$.

A monomial $Sq^{i_1}\cdots Sq^{i_k}$, the composition of the individual operations $Sq^{i_j}$ for $1 \le j \le k$, is denoted by $Sq^{I}$, where $I = (i_1, \cdots, i_k)$.
Let $d(I) = \Sigma_{j=1}^{k} i_j$ denote the degree of $Sq^{I}$. 
The operation $Sq^{I}$ is called admissible if $i_j \ge 2i_{j+1}$ for each $j$.
The excess of an admissible $Sq^{I}$ is defined as 
\[ e(I)=\Sigma_j(i_j - 2i_{j+1}). \] 
Then the mod $2 $ cohomology ring $H^{\ast}(K(\Z/2, n))$ can be described as follows (cf. Hatcher \cite[Theorem 5.32]{hat-bss}).
\begin{lemma}\label{lem:mac}
$H^{\ast}(K(\Z/2, n))$ is the polynomial ring $\Z/2[Sq^{I}(l_n)]$, where $l_n$ is the fundamental class of $H^{n}(K(\Z/2,n))$ and $I$ ranges over all admissible sequence with excess $e(I) < n$.
\end{lemma}

Finally, from the cohomology Serre spectral sequence (cf. \cite[p.68, Proposition 3.2.1]{kochman} or \cite[p.145, Example 5.D]{mccleary}), we have
\begin{lemma}[Serre Long Exact Cohomology Sequence]
\label{lem:serre}
Let $F \xrightarrow{i} E \xrightarrow{\pi} B$ be a fibration where $B$ is $(m-1)$-connected ($m \ge 2$) and $F$ is $(n-1)$-connected ($n \ge 1$).
For any abelian group $G$ and $p = m+n-1$, there is a long exact sequence:
\begin{align*}
H^1(E;G) \xrightarrow{i^{\ast}} H^1(F;G) \xrightarrow{\tau} H^2(B;G) \xrightarrow{\pi^{\ast}} \cdots \xrightarrow{\pi^{\ast}} H^{p}(E;G) \xrightarrow{i^{\ast}} H^{p}(F;G),
\end{align*}
where $\tau$ is the transgression. 
\end{lemma}

\begin{remark}
It is known that the Steenrod Squares $Sq^i$, $i \ge 0$, commute with the transgression $\tau$.
\end{remark}

%{s:main}
%%%%%%%%%%%%%%%%%%%%%%%%%%%%%%%%%%%%%%%%%%%%%%%%%%%%%%%%

\subsection{Mod $2$ Cohomology Groups of $C_{\varPsi}$}
\label{ss:cofi}
%%%%%%%%%%%%%%%%%%%%%%%%%%%%%%%%%%%%%%%%%%%%%%%%%%%%%%%
%For large $r$, let $\varPsi \colon \Sigma^{r-4n}K(\Z/2, 4n) \rightarrow K(\Z/2, r)$
%denote the map for which 
%\[ \varPsi^{\ast} l_{r} = \sigma^{r-4n} l_{4n},\]
%where $\sigma^{k}$ means the $k$-fold composition of $\sigma$.

This subsection analyzes the cofibration \eqref{cof:Psi} to determine the mod $2$ cohomology groups of $C_{\varPsi}$ up to dimension  $r+4m+9$.
%\begin{align} \label{cof:cpsi}
%\Sigma^{r-4n}K(\Z/2, 4n) &\xrightarrow{\varPsi} K(\Z/2,r) \xrightarrow{\pi_{\varPsi}} C_{\varPsi}.
%\end{align}

\begin{lemma}\label{lem:cpsi}
$C_{\varPsi}$ is $(r{+}4m)$-connected, and $\pi_{r+4m+1}(C_{\varPsi}) \cong \Z/2$.
\end{lemma}

\begin{proof}
The $(r{+}4m)$-connectivity of $C_{\varPsi}$ follows directly from its construction and the Freudenthal suspension theorem (Lemma \ref{lem:freu}).

Since $C_{\varPsi}$ is $(r{+}4m)$-connected, the Freudenthal suspension theorem implies that 
\[ \pi_{r+4m+1}(C_{\varPsi}) \cong \pi_{r+4m+1}^s (C_{\varPsi}), \] 
where $\pi_{r+4m+1}^s (C_{\varPsi})$ is the $(r{+}4m{+}1)$-th stable homotopy group of $C_{\varPsi}$.
The exact sequence of stable homotopy groups for the cofibration \eqref{cof:Psi} yields
\[ \pi_{r+4m+1}^s (C_{\varPsi}) \cong \pi_{r+4m}^s( \Sigma^{r-4m} K(\Z/2, 4m) ).  \]
According to Brown \cite[Lemma (1.2)]{brn72}, 
\[ \pi_{r+4m}^s( \Sigma^{r-4m} K(\Z/2, 4m) ) \cong \pi_{8m}^s (K(\Z/2, 4m)) \cong \Z/2, \]
which completes the proof.
\end{proof}

%\subsection{The $\bmod ~2$ cohomology groups $\widetilde{H}^{\ast}(C_{\varPsi}; \Z/2)$ of $C_{\varPsi}$}

Consider the exact sequence in mod $2$ cohomology induced by the cofibration \eqref{cof:Psi}:
\begin{align}\label{exact:Psi}
\cdots \to \widetilde{H}^{\ast}(C_{\varPsi}) \xrightarrow{\pi_{\varPsi}^{\ast}} \widetilde{H}^{\ast}(K(\Z/2, r)) \xrightarrow{\varPsi^{\ast}} \widetilde{H}^{\ast} (\Sigma^{r-4m}K(\Z/2, 4m)) \xrightarrow{\delta} \widetilde{H}^{\ast+1}(C_{\varPsi}) \to \cdots.
\end{align}
Let $ (\Im \pi_{\varPsi}^{\ast})^{+j}$ denote the image of 
\[ \pi_{\varPsi}^{\ast} \colon \widetilde{H}^{r+4m+j}(C_{\varPsi}) \to \widetilde{H}^{r+4m+j}(K(\Z/2, r)), \]
let $( \mathrm{Ker} \varPsi^{\ast} )^{+j}$ denote the kernel of 
\[ \varPsi^{\ast} \colon \widetilde{H}^{r+4m+j}(K(\Z/2, r)) \to \widetilde{H}^{r+4m+j} (\Sigma^{r-4n}K(\Z/2, 4m)),\]
and let $( \Im \delta )^{+j}$ denote the image of 
\[ \delta \colon \widetilde{H}^{r + 4m + j - 1} (\Sigma^{r-4m}K(\Z/2, 4m)) \to \widetilde{H}^{r + 4m + j}(C_{\varPsi}).\]
From the exact sequence \eqref{exact:Psi}, we have $ (\Im \pi_{\varPsi}^{\ast})^{+j} = ( \mathrm{Ker} \varPsi^{\ast} )^{+j}$ and 
\begin{equation}\label{eq:hcPsi}
 \widetilde{H}^{r+4m+j}(C_{\varPsi}) \cong (\Im \pi_{\varPsi}^{\ast})^{+j} \oplus ( \Im \delta )^{+j}  = ( \mathrm{Ker} \varPsi^{\ast} )^{+j} \oplus ( \Im \delta )^{+j}.
 \end{equation}
Since $r$ is large, for fixed $m$ and $j \le 9$, the group $H^{r+4m+j}(K(\Z/2, r))$ has a basis given by the classes $Sq^{I}l_r$ with $I$ admissible and $d(I) = 4m{+}j$.
Because the Steenrod Squares commute with the suspension isomorphism $\sigma$, we have
\begin{equation}\label{eq:sqlr}
\varPsi^{\ast} ( Sq^{I} l_r ) = \sigma^{r-4m} Sq^{I} l_{4m}.
\end{equation}
Thus, $(\Im \pi_{\varPsi}^{\ast})^{+j} = ( \mathrm{Ker} \varPsi^{\ast} )^{+j}$ has a basis given by those $Sq^{I} l_r$ with $I$ admissible, $d(I) = 4m+j$, and $e(I) > 4m$.

Furthermore, assuming $m \ge 2$,  for $j \le 9$, 
the group $( \Im \delta )^{+j}$ (isomorphic to the cokernel of $\varPsi^{\ast}$) has a basis given by classes $\delta \sigma Sq^{I_1}l_{4m}  Sq^{I_2}l_{4m}$, where $I_1$ and $I_2$ are admissible sequences with $d(I_1) + d(I_2) = j-1$, $e(I_1 )< 4m$, $e(I_2) < 4m$, and $I_1 \neq I_2 $. 
Here, $\sigma$ denotes $\sigma^{r-4m}$, the $(r{-}4m)$-fold suspension isomorphism.
 (Note: if $m = 2$, the element $\delta \sigma l_{4m}^3$ should be added to the basis of  $( \Im \delta )^{+9}$, 
 but since it does not affect the subsequent calculation of $\wt{\Omega}^{\mathrm{Spin}^c}_{r+4m+7}(C_{\varPsi})$, we omit it and consider $( \Im \delta )^{+9}$ generated only by the classes $\delta \sigma Sq^{I_1}l_{4m}  Sq^{I_2}l_{4m}$.)

Using the isomorphisms \eqref{eq:hcPsi} and the basis descriptions above, the mod $2$ cohomology groups $\widetilde{H}^{r+4m+j}(C_{\varPsi})$ for $j \le 9$ can be determined.
However, to simplify the calculation of $\Omega_{r+4m+7}^{\mathrm{Spin}^c} (C_{\varPsi})$, it is useful to modify the basis.

For the groups $(\Im \pi_{\varPsi}^{\ast})^{+j} = ( \mathrm{Ker} \varPsi^{\ast} )^{+j}$ with $j \le 9$, define
\begin{equation*}
\alpha_{2^j}  = Sq^{ 4m + 2^j } l_r , ~ \text{for}~0 \le j \le 3.
\end{equation*}
Let $\mathscr{A}$ be the mod $2$ Steenrod algebra. 
Using the Adem relations \eqref{eq:adem}, a straightforward calculation shows that, through dimension $r + 4m + 9$, $\Im \pi_{\varPsi}^{\ast} =  \mathrm{Ker} \varPsi^{\ast} $ is an $\mathscr{A}$-module generated by $\alpha_{1}$, $\alpha_{2}$, $\alpha_{4}$, and $\alpha_{8}$, subject to the relations:
\begin{align}
Sq^1 \alpha_1 & = 0, \label{eq:a1}\\
Sq^3 \alpha_1 & = Sq^2 \alpha_2, \label{eq:a1a2} \\
Sq^4 \alpha_4 & = \delta_m \alpha_8 + Sq^6 \alpha_2 + Sq^7 \alpha_1, \label{eq:a4}
\end{align}
where $\delta_m = 0$ if $m$ is even, and $\delta_m = 1$ if $m$ is odd.
From these relations, and the Adem relations \eqref{eq:adem} $ Sq^1 Sq^{2k} = Sq^{2k+1}$, $Sq^1 Sq^{2k+1} = 0$, $Sq^2 Sq^2 = Sq^3 Sq^1$ and $Sq^2 Sq^3 = Sq^5 + Sq^4 Sq^1$,
%$Sq^4 Sq^3 = Sq^5 Sq^2$ 
%by the Adem relation \eqref{eq:adem}, we must have 
%$Sq^3 \alpha_2 = 0$, 
%$Sq^5 \alpha_2  = Sq^4 Sq^1 \alpha_2$, $Sq^5 Sq^1 \alpha_2  = 0$,
%$Sq^4 Sq^2 \alpha_2  = Sq^5 Sq^2 \alpha_1$, $Sq^5 Sq^2 \alpha_2  = 0$, 
%and
we also obtain: 
\begin{align} 
Sq^5 \alpha_1  & = Sq^3 Sq^1 \alpha_2,  \label{eq:sq5a1} \\
Sq^5 \alpha_4  & = \delta_m Sq^1 \alpha_8 + Sq^7 \alpha_2. \label{eq:sq5a4}
\end{align} 
%Moreover, the basis of $ (\Im \pi_{\varPsi}^{\ast})^{+j} = ( \mathrm{Ker} \varPsi^{\ast} )^{+j}$, $ j \le 9$, can be chosen as:
%\begin{enumerate}
%\item[]  $(\Im \pi_{\varPsi}^{\ast})^{+1} \cong \Z/2$ generated by  $\alpha_1$, 
%\item[] $(\Im \pi_{\varPsi}^{\ast})^{+2} \cong \Z/2$  generated by  $\alpha_2$, 
%\item[]  $(\Im \pi_{\varPsi}^{\ast})^{+3} \cong (\Z/2)^2$ generated by $Sq^2 \alpha_1$, and $Sq^1 \alpha_2$, 
%\item[] $(\Im \pi_{\varPsi}^{\ast})^{+4} \cong (\Z/2)^2$ generated by $Sq^3 \alpha_1$, and $\alpha_4$, 
%\item[] $(\Im \pi_{\varPsi}^{\ast})^{+5} \cong (\Z/2)^3$ generated by $Sq^4 \alpha_1$, $Sq^2 Sq^1 \alpha_2$, and $Sq^1 \alpha_4$, 
%\item [] $(\Im \pi_{\varPsi}^{\ast})^{+6} \cong (\Z/2)^3$  generated by  $Sq^5 \alpha_1$, $Sq^4 \alpha_2$, and $Sq^2 \alpha_4$, 
%\item[] $(\Im \pi_{\varPsi}^{\ast})^{+7} \cong (\Z/2)^5$ generated by $Sq^6 \alpha_1$, $Sq^4 Sq^2 \alpha_1$, $Sq^5 \alpha_2$, $Sq^3 \alpha_4$, and $Sq^2 Sq^1 \alpha_4$, 
%\item[] $(\Im \pi_{\varPsi}^{\ast})^{+8} \cong (\Z/2)^5$ generated by $Sq^7 \alpha_1$, $Sq^5 Sq^2 \alpha_1$, $Sq^6 \alpha_2$, $Sq^3 Sq^1 \alpha_4$, and $\alpha_8$, 
%\item[] $(\Im \pi_{\varPsi}^{\ast})^{+9} \cong (\Z/2)^7$  generated by $Sq^8 \alpha_1$, $Sq^6 Sq^2 \alpha_1$, $Sq^7 \alpha_2$, $Sq^6 Sq^1 \alpha_2$, $Sq^4 Sq^2 Sq^1 \alpha_2$, $Sq^4 Sq^1 \alpha_4$, and $Sq^1 \alpha_8$.
%\end{enumerate} 
The basis of $ (\Im \pi_{\varPsi}^{\ast})^{+j} = ( \mathrm{Ker} \varPsi^{\ast} )^{+j}$, $ j \le 9$ is listed in Table $1$.
\begin{center}
\begin{tabular}{|c|c|l|} \multicolumn{3}{c} {Table $1$. Basis of $ (\Im \pi_{\varPsi}^{\ast})^{+j} = ( \mathrm{Ker} \varPsi^{\ast} )^{+j}$ } \\
\hline
\multicolumn{1}{|c}{\rule{0pt}{14pt} $j$} &
\multicolumn{1}{|c}{$(\Im \pi_{\varPsi}^{\ast})^{+j} $} &
\multicolumn{1}{|c|}{Basis}  \\
%$j$ & $ \widetilde{H}^{r+4n+j}(C_{\varPsi}; \Z/2) $ & generators of $\widetilde{H}^{r+4n+j}(C_{\varPsi}; \Z/2)$ \\
\hline
\rule{0pt}{14pt} $1$ & $ \Z/2$ &  $\alpha_1$ \\
\hline
\rule{0pt}{14pt} $2$ & $ \Z/2$ &  $\alpha_2$ \\
\hline
\rule{0pt}{14pt} $3$ & $ (\Z/2)^2$ &  $Sq^2 \alpha_1$, $Sq^1 \alpha_2$, \\
\hline
\rule{0pt}{14pt} $4$ & $( \Z/2)^2$ & $Sq^3 \alpha_1$, $\alpha_4$ \\
\hline
\rule{0pt}{14pt} $5$ & $ (\Z/2)^3$ & $Sq^4 \alpha_1$, $Sq^2 Sq^1 \alpha_2$, $Sq^1 \alpha_4$ \\
\hline
\rule{0pt}{14pt} $6$ & $ (\Z/2)^3$ & $Sq^5 \alpha_1$, $Sq^4 \alpha_2$, $Sq^2 \alpha_4$\\
\hline
\rule{0pt}{14pt} $7$ & $ (\Z/2)^{5}$ & $Sq^6 \alpha_1$, $Sq^4 Sq^2 \alpha_1$, $Sq^5 \alpha_2$, $Sq^3 \alpha_4$, $Sq^2 Sq^1 \alpha_4$,     \\  
\hline
\rule{0pt}{14pt} $8$ &  $(\Z/2)^{5}$ & $Sq^7 \alpha_1$, $Sq^5 Sq^2 \alpha_1$, $Sq^6 \alpha_2$, $Sq^3 Sq^1 \alpha_4$, $\alpha_8$ \\
\hline
\rule{0pt}{14pt} $9$ & $(\Z/2)^{7}$ & $Sq^8 \alpha_1$, $Sq^6 Sq^2 \alpha_1$, $Sq^7 \alpha_2$, $Sq^6 Sq^1 \alpha_2$, $Sq^4 Sq^2 Sq^1 \alpha_2$, $Sq^4 Sq^1 \alpha_4$, $Sq^1 \alpha_8$ \\
\hline
\end{tabular}
\end{center}
\vspace{5pt}

For the groups $( \Im \delta )^{+j}$ with $j \le 9$, define
\begin{align*}
\gamma_j & = \delta \sigma l_{4m} Sq^{j-1} l_{4m},  & \text{for}~ 2 \le  j \le 9, \\ 
\gamma_{j1} & = \delta \sigma l_{4m} Sq^{j-2} Sq^1 l_{4m}, & \text{for}~ 7 \le  j \le 9.
\end{align*}
Since the Steenrod squares commute with $\sigma$ and $\delta$, and since
\[ \delta \sigma Sq^I l_{4m} Sq^I l_{4m} = \delta \sigma \Sq^{d(I) + 4m} Sq^I  l_{4m} = \Sq^{d(I) + 4m} Sq^I \delta \varPsi^{\ast} l_r = 0 \]
for any $I = (i_1, \cdots, i_k)$, it follows from the Cartan formula \eqref{eq:cartan} and the Adem relations \eqref{eq:adem} that
\begin{align}
Sq^1 \gamma_2 & = 0, \label{eq:sq1g2}\\
Sq^3 Sq^1 \gamma_3 & = Sq^5 \gamma_2, \label{eq:sq5g2}\\
Sq^5 Sq^1 \gamma_3  & = 0. \label{eq:sq51g3}
\end{align} 
%Moreover, the basis of $ (\Im \pi_{\varPsi}^{\ast})^{+j} = ( \mathrm{Ker} \varPsi^{\ast} )^{+j}$, $ j \le 9$, can be chosen as:
%\begin{enumerate}
%\item[]  $(\Im \pi_{\varPsi}^{\ast})^{+1} \cong \Z/2$ generated by  $\alpha_1$, 
%\item[] $(\Im \pi_{\varPsi}^{\ast})^{+2} \cong \Z/2$  generated by  $\alpha_2$, 
%\item[]  $(\Im \pi_{\varPsi}^{\ast})^{+3} \cong (\Z/2)^2$ generated by $Sq^2 \alpha_1$, and $Sq^1 \alpha_2$, 
%\item[] $(\Im \pi_{\varPsi}^{\ast})^{+4} \cong (\Z/2)^2$ generated by $Sq^3 \alpha_1$, and $\alpha_4$, 
%\item[] $(\Im \pi_{\varPsi}^{\ast})^{+5} \cong (\Z/2)^3$ generated by $Sq^4 \alpha_1$, $Sq^2 Sq^1 \alpha_2$, and $Sq^1 \alpha_4$, 
%\item [] $(\Im \pi_{\varPsi}^{\ast})^{+6} \cong (\Z/2)^3$  generated by  $Sq^5 \alpha_1$, $Sq^4 \alpha_2$, and $Sq^2 \alpha_4$, 
%\item[] $(\Im \pi_{\varPsi}^{\ast})^{+7} \cong (\Z/2)^5$ generated by $Sq^6 \alpha_1$, $Sq^4 Sq^2 \alpha_1$, $Sq^5 \alpha_2$, $Sq^3 \alpha_4$, and $Sq^2 Sq^1 \alpha_4$, 
%\item[] $(\Im \pi_{\varPsi}^{\ast})^{+8} \cong (\Z/2)^5$ generated by $Sq^7 \alpha_1$, $Sq^5 Sq^2 \alpha_1$, $Sq^6 \alpha_2$, $Sq^3 Sq^1 \alpha_4$, and $\alpha_8$, 
%\item[] $(\Im \pi_{\varPsi}^{\ast})^{+9} \cong (\Z/2)^7$  generated by $Sq^8 \alpha_1$, $Sq^6 Sq^2 \alpha_1$, $Sq^7 \alpha_2$, $Sq^6 Sq^1 \alpha_2$, $Sq^4 Sq^2 Sq^1 \alpha_2$, $Sq^4 Sq^1 \alpha_4$, and $Sq^1 \alpha_8$.
%\end{enumerate} 
Through dimension $r + 4m + 9$, $ \Im \delta $ is an $\mathscr{A}$-module generated by $\gamma_j$ ($2 \le  j \le 9$) and $\gamma_{j1}$ ($7 \le  j \le 9$), subject to relations \eqref{eq:sq1g2}-\eqref{eq:sq51g3}.
The basis of $(\Im \delta)^{+j} $ for $j \le 9$ is listed in Table $2$.
\begin{center}
\begin{tabular}{|c|c|l|} \multicolumn{3}{c} {Table $2$. Basis of $(\Im \delta)^{+j}$ } \\
\hline
\multicolumn{1}{|c}{\rule{0pt}{14pt} $j$} &
\multicolumn{1}{|c}{$(\Im \delta)^{+j}$} &
\multicolumn{1}{|c|}{Basis}  \\
%$j$ & $ \widetilde{H}^{r+4n+j}(C_{\varPsi}; \Z/2) $ & generators of $\widetilde{H}^{r+4n+j}(C_{\varPsi}; \Z/2)$ \\
\hline
\rule{0pt}{14pt} $\le 1$ & $0$ &  \\
\hline
\rule{0pt}{14pt} $2$ & $ \Z/2$ &  $\gamma_2$ \\
\hline
\rule{0pt}{14pt} $3$ & $ \Z/2$ &  $\gamma_3$ \\
\hline
\rule{0pt}{14pt} $4$ & $( \Z/2)^3$ & $Sq^2 \gamma_2$, $Sq^1 \gamma_3$, $\gamma_4$ \\
\hline
\rule{0pt}{14pt} $5$ & $ (\Z/2)^4$ & $Sq^3 \gamma_2$, $Sq^2 \gamma_3$, $Sq^1 \gamma_4$, $\gamma_5$ \\
\hline
\rule{0pt}{14pt} $6$ & $ (\Z/2)^6$ & $Sq^4 \gamma_2$, $Sq^3 \gamma_3$, $Sq^2 Sq^1 \gamma_3$, $Sq^2 \gamma_4$, $Sq^1 \gamma_5$, $\gamma_6$ \\
\hline
\rule{0pt}{14pt} $7$ & $ (\Z/2)^{8}$ & $Sq^5 \gamma_2$, $Sq^4 \gamma_3$, $Sq^3 \gamma_4$, $Sq^2 Sq^1 \gamma_4$, $Sq^2 \gamma_5$, $Sq^1 \gamma_6$, $\gamma_{7}$, $\gamma_{71}$    \\  
\hline
\rule{0pt}{14pt} $8$ &  $(\Z/2)^{13}$ & $Sq^6 \gamma_2$, $Sq^4 Sq^2 \gamma_2$, $Sq^5 \gamma_3$, $Sq^4 Sq^1 \gamma_3$, $Sq^4 \gamma_4$, $Sq^3 Sq^1 \gamma_4$, 
$Sq^3 \gamma_5$, $Sq^2 Sq^1 \gamma_5$, 
\\ & & $Sq^2 \gamma_6$, $Sq^1 \gamma_{7}$, $Sq^1 \gamma_{71}$, $\gamma_{8}$, $\gamma_{81}$ \\
\hline
\rule{0pt}{14pt} $9$ & $(\Z/2)^{16}$ & $Sq^7 \gamma_2$, $Sq^5 Sq^2 \gamma_2$, $Sq^6 \gamma_3$, $Sq^4 Sq^2 \gamma_3$, $Sq^5 \gamma_4$, $Sq^4 Sq^1 \gamma_4$,
$Sq^4 \gamma_5$, $Sq^3 Sq^1 \gamma_5$, 
\\ & & $Sq^3 \gamma_6$, $Sq^2 Sq^1 \gamma_6$, $Sq^2 \gamma_{7}$, $Sq^2 \gamma_{71}$, $Sq^1 \gamma_{8}$, $Sq^1 \gamma_{81}$, $\gamma_{9}$, $\gamma_{91}$ \\
\hline
\end{tabular}
\end{center}
\vspace{5pt}

%\begin{lemma}\label{lem:imdelta}
%The basis of $(\Im \delta)^{+j} $, $j \le 9$, can be chosen as follows
%\begin{enumerate}
%\item[] $(\Im \delta)^{+2} \cong \Z/2 $ generated by  $\gamma_2$,
%\item[] $(\Im \delta)^{+3} \cong \Z/2$ generated by  $\gamma_3$,
%\item[] $(\Im \delta)^{+4} \cong (\Z/2)^3$ generated by $Sq^2 \gamma_2$, $Sq^1 \gamma_3$, $\gamma_4$,
%\item[] $(\Im \delta)^{+5} \cong (\Z/2)^4$ generated by $Sq^3 \gamma_2$, $Sq^2 \gamma_3$, $Sq^1\gamma_4$, $\gamma_5$,
%\item[] $(\Im \delta)^{+6} \cong (\Z/2)^6$ generated by $Sq^4 \gamma_2$, $Sq^3 \gamma_3$, $Sq^2 Sq^1 \gamma_3$, $Sq^2 \gamma_4$, $Sq^1 \gamma_5$, $\gamma_6$,
%\item[] $(\Im \delta)^{+7} \cong (\Z/2)^8$ generated by $Sq^5 \gamma_2$, $Sq^4 \gamma_3$, $Sq^3 \gamma_4$, $Sq^2 Sq^1 \gamma_4$, $Sq^2 \gamma_5$, $Sq^1 \gamma_6$, $\gamma_7$, $\gamma_{71}$,
%\item[] $(\Im \delta)^{+8} \cong (\Z/2)^{13}$ generated by $Sq^6 \gamma_2$, $Sq^4 Sq^2 \gamma_2$, $Sq^5 \gamma_3$, $Sq^4 Sq^1 \gamma_3$, $Sq^4 \gamma_4$, $Sq^3 Sq^1 \gamma_4$, $Sq^3 \gamma_5$, $Sq^2 Sq^1 \gamma_5$, $Sq^2 \gamma_6$, $Sq^1 \gamma_7$,
%$Sq^1 \gamma_{71}$, $\gamma_8$, $\gamma_{81}$, 
%\item[] $(\Im \delta)^{+9} \cong (\Z/2)^{16}$  generated by  $Sq^7 \gamma_2$, $Sq^5 Sq^2 \gamma_2$, $Sq^6 \gamma_3$, $Sq^4 Sq^2 \gamma_3$, $Sq^5 \gamma_4$, $Sq^4 Sq^1 \gamma_4$, $Sq^4 \gamma_5$, $Sq^3 Sq^1 \gamma_5$, $Sq^3 \gamma_6$, $Sq^2 Sq^1 \gamma_6$, $Sq^2 \gamma_{7}$, $Sq^2 \gamma_{71}$, $Sq^1 \gamma_{8}$, $Sq^1 \gamma_{81}$, $\gamma_{9}$, $\gamma_{91}$.
%\end{enumerate}
%
%\end{lemma}

Based on the isomorphisms \eqref{eq:hcPsi} and the basis descriptions in Tables $1$ and $2$, the mod $2$ cohomology groups  $\widetilde{H}^{+j}(C_{\varPsi}) $ for $j \le 9$ and their bases can be summarized as follows.

Let $\overline{ \alpha_{2^j} } \in H^{\ast}(C_{\varPsi})$ denote an element satisfying 
\[ \pi_{\varPsi}^{\ast}( \overline{ \alpha_{2^j} } ) = \alpha_{2^j} \in H^{\ast}(K(\Z/2,r)). \] 
Let $\widetilde{H}^{+j}(C_{\varPsi})$ denote the $(r{+}4m{+}j)$-th mod $2$ cohomology group of $C_{\varPsi}$.
The groups  $\widetilde{H}^{+j}(C_{\varPsi}) $ for $j \le 9$ and their bases are listed in Table $3$.

\begin{center}
\begin{tabular}{|c|c|l|}\multicolumn{3}{c} { Table $3$. Mod $2$ Cohomology Groups of $C_{\varPsi}$ } \\
\hline
\multicolumn{1}{|c}{\rule{0pt}{14pt} $j$} &
\multicolumn{1}{|c}{$ \widetilde{H}^{+j}(C_{\varPsi}) $} &
\multicolumn{1}{|c|}{Basis of $\widetilde{H}^{+j}(C_{\varPsi})$}  \\
%$j$ & $ \widetilde{H}^{r+4n+j}(C_{\varPsi}; \Z/2) $ & generators of $\widetilde{H}^{r+4n+j}(C_{\varPsi}; \Z/2)$ \\
\hline
\rule{0pt}{14pt} $1$ & $\Z/2$ & $\overline{ \alpha_1 }$ \\
\hline
\rule{0pt}{14pt} $2$ & $ (\Z/2)^2$ & $\overline{ \alpha_2 }$, $\gamma_2$ \\
\hline
\rule{0pt}{14pt} $3$ & $ (\Z/2)^3$ & $ Sq^2 \overline{ \alpha_1 }$, $ Sq^1 \overline{ \alpha_2 }$, $\gamma_3$ \\
\hline
\rule{0pt}{14pt} $4$ & $ (\Z/2)^5$ & $ Sq^3 \overline{ \alpha_1 } $,  $\overline{ \alpha_4 }$, $Sq^2 \gamma_2$,~$Sq^1 \gamma_3$,~$\gamma_4$ \\
\hline
\rule{0pt}{14pt} $5$ & $ (\Z/2)^7$ & $ Sq^4 \overline{ \alpha_1 } $,  $Sq^2 Sq^1 \overline{ \alpha_2 }$,  $ Sq^1 \overline{ \alpha_4 }$, $Sq^3 \gamma_2$,~$Sq^2 \gamma_3$,~$Sq^1\gamma_4$,~$\gamma_5$ \\
\hline
\rule{0pt}{14pt} $6$ & $ (\Z/2)^9$ & $ Sq^5 \overline{ \alpha_1 } $,  $Sq^4 \overline{ \alpha_2 }$,  $ Sq^2 \overline{ \alpha_4 }$, $Sq^4 \gamma_2$, $Sq^3 \gamma_3$, $Sq^2 Sq^1 \gamma_3$, $Sq^2 \gamma_4$, $Sq^1 \gamma_5$, $\gamma_6$ \\
\hline
\rule{0pt}{14pt} $7$ & $ (\Z/2)^{13}$ & $ Sq^6 \overline{ \alpha_1 } $, $Sq^4 Sq^2 \overline{ \alpha_1 }$, $Sq^5 \overline{ \alpha_2 }$,  $ Sq^3 \overline{ \alpha_4 }$, $Sq^2 Sq^1 \overline{ \alpha_4 }$,
\\ && $Sq^5 \gamma_2$, $Sq^4 \gamma_3$, $Sq^3 \gamma_4$, $Sq^2 Sq^1 \gamma_4$, $Sq^2 \gamma_5$, $Sq^1 \gamma_6$, $\gamma_7$, $\gamma_{71}$ \\
\hline
\rule{0pt}{14pt} $8$ &  $(\Z/2)^{18}$ & $Sq^7 \overline{ \alpha_1 } $, $ Sq^5 Sq^2 \overline{ \alpha_1} $, $ Sq^6 \overline{ \alpha_2}$, $Sq^3 Sq^1 \overline{ \alpha_4 } $, $\overline{ \alpha_8} $,
\\ & &  $Sq^6 \gamma_2$, $Sq^4 Sq^2 \gamma_2$, $Sq^5 \gamma_3$, $Sq^4 Sq^1 \gamma_3$, $Sq^4 \gamma_4$, $Sq^3 Sq^1 \gamma_4$, 
$Sq^3 \gamma_5$, $Sq^2 Sq^1 \gamma_5$, 
\\ & & $Sq^2 \gamma_6$, $Sq^1 \gamma_7$, $Sq^1 \gamma_{71}$, $\gamma_8$, $\gamma_{81}$ \\
\hline
\rule{0pt}{14pt} $9$ & $(\Z/2)^{23}$ & $Sq^8 \overline{ \alpha_1 } $, $Sq^6 Sq^2 \overline{ \alpha_1 } $, $Sq^7 \overline{ \alpha_2 } $, $Sq^6 Sq^1 \overline{ \alpha_2 } $, $ Sq^4 Sq^2 Sq^1 \overline{ \alpha_2 } $, $Sq^4 Sq^1 \overline{ \alpha_4 } $, $Sq^1 \overline{ \alpha_8 } $,
\\ & & $Sq^7 \gamma_2$, $Sq^5 Sq^2 \gamma_2$, $Sq^6 \gamma_3$, $Sq^4 Sq^2 \gamma_3$, $Sq^5 \gamma_4$, $Sq^4 Sq^1 \gamma_4$, $Sq^4 \gamma_5$, $Sq^3 Sq^1 \gamma_5$, 
\\ & & $Sq^3 \gamma_6$, $Sq^2 Sq^1 \gamma_6$, $Sq^2 \gamma_{7}$, $Sq^2 \gamma_{71}$, $Sq^1 \gamma_{8}$, $Sq^1 \gamma_{81}$, $\gamma_{9}$, $\gamma_{91}$ \\
\hline
\end{tabular}
\end{center}
\hspace{5pt}

Moreover, through dimension $r + 4m + 9$, $ \widetilde{H}^{\ast}(C_{\varPsi})$ is an $\mathscr{A}$ module generated by $\overline{ \alpha_{2^j}}$ ($0 \le j \le 3$), $\gamma_j$ ($2 \le  j \le 9$), and $\gamma_{j1} $ ($7 \le  j \le 9$), subject to Relations \eqref{eq:sq1g2} - \eqref{eq:sq51g3} and the following additional relations.

\begin{lemma}
\label{lem:sq1a1}
$Sq^1 \overline{ \alpha_1 } = \gamma_2$.
\end{lemma}

\begin{proof}
By Identity \eqref{eq:a1}, $Sq^1  \alpha_1 = 0$ in $H^{\ast}(K(\Z/2, r))$, so $Sq^1 \overline{ \alpha_1 } $ must lie in $( \Im \delta )^{+2}$, which is isomorphic to $\Z/2$ and generated by $\gamma_2$ (Table $2$).
Thus, it suffices to prove that $Sq^1 \overline{ \alpha_1 } \neq 0$.

Consider the Bockstein sequence for $C_{\varPsi}$ associated to the coefficient sequence $\Z/2 \xrightarrow{ \times 2} \Z/4 \to \Z/2$:
\begin{equation*}
\cdots \rightarrow H^{r+4m+1}(C_{\varPsi};\Z/4) \xrightarrow{ \rho_2 } H^{r+4m+1}(C_{\varPsi}) \xrightarrow{Sq^1} H^{r+4m+2}(C_{\varPsi})  \to \cdots.
\end{equation*}
Since $C_{\varPsi}$ is $(r{+}4m)$-connected and $\pi_{r+4m+1}(C_{\varPsi}) \cong \Z/2$ by Lemma \ref{lem:cpsi}, the homomorphism $\rho_2 \colon H^{r+4m+1}(C_{\varPsi};\Z/4) \to H^{r+4m+1}(C_{\varPsi})$ is zero. 
Therefore,
\[ Sq^1 \colon H^{r+4m+1}(C_{\varPsi}) \to H^{r+4m+2}(C_{\varPsi}) \]
 is injective, completing the proof.
\end{proof}

\begin{lemma}\label{lem:sq3a1}
The generator $\overline{ \alpha_2 }$ can be chosen such that
\begin{align}\label{eq:sq3a1}
Sq^3 \overline{ \alpha_1} + Sq^2 \overline{ \alpha_2}  & =  Sq^1 \gamma_3.
\end{align}

\end{lemma}

\begin{proof}
By Identity \eqref{eq:a1a2}, $Sq^3 \alpha_1 + Sq^2 \alpha_2 = 0$ in $H^{\ast}(K(\Z/2, r))$,
so $Sq^3 \overline{ \alpha_1 } + Sq^2 \overline{ \alpha_2 }$ must lie in $( \Im \delta )^{+4}$, which is isomorphic to $(\Z/2)^3$ and generated by $Sq^2 \gamma_2$, $Sq^1 \gamma_3$, and $\gamma_4$ (Table $2$). 
Assume
\begin{align} \label{eq:sq3a1x}
Sq^3 \overline{ \alpha_1 } + Sq^2 \overline{ \alpha_2 } = x Sq^2 \gamma_2 + y Sq^1 \gamma_3 + z \gamma_4,
\end{align}
where $x,~y,~z \in \Z/2$.

Similarly, by Relation \eqref{eq:sq5a1} and the Adem relation $Sq^1 Sq^{2k+1} =0$, we have $Sq^5 \alpha_1 + Sq^3 Sq^1 \alpha_2  {=}  0$ in $H^{\ast}(K(\Z/2, r))$.
The element $Sq^5 \overline{ \alpha_1 } + Sq^3 Sq^1 \overline{ \alpha_2 } $ must lie both in $(\Im \delta)^{+6}$ and in the kernel of $Sq^1$,
so it is a linear combination of $Sq^4 \gamma_2 + Sq^2 Sq^1 \gamma_3$, $Sq^3 \gamma_3$, and $Sq^1 \gamma_5$.
Assume
\begin{align}\label{eq:sq5alpha}
Sq^5 \overline{ \alpha_1 } + Sq^3 Sq^1 \overline{ \alpha_2}  & =   a ( Sq^4 \gamma_2 + Sq^2 Sq^1 \gamma_3 )  + b  Sq^3 \gamma_3 + c Sq^1 \gamma_5,
\end{align}
where $a,~b,~c \in \Z/2$.

From the Adem relations $Sq^2 Sq^3 = Sq^5 + Sq^4 Sq^1$ and $Sq^2 Sq^2 = Sq^3 Sq^1$, applying $Sq^2$ to the left-hand side of \eqref{eq:sq3a1x} and using Identity \eqref{eq:sq5alpha} and Lemma \ref{lem:sq1a1} gives
\begin{align*}
Sq^2 ( Sq^3 \overline{ \alpha_1 } + Sq^2 \overline{ \alpha_2 } ) & = Sq^5 \overline{ \alpha_1 }+ Sq^3 Sq^1 \overline{ \alpha_2 } + Sq^4 Sq^1 \overline{ \alpha_1 } \\
& = (a +1) Sq^4 \gamma_2 + a Sq^2 Sq^1 \gamma_3  + b  Sq^3 \gamma_3 + c Sq^1 \gamma_5.
\end{align*}
On the other hand, applying $Sq^2$ to the right-hand side of \eqref{eq:sq3a1x} and using Equation \eqref{eq:sq1g2} yields:
\begin{align*}
Sq^2 ( x Sq^2 \gamma_2 + y Sq^1 \gamma_3 + z \gamma_4 ) & =  y Sq^2 Sq^1 \gamma_3 + z Sq^2 \gamma_4.
\end{align*}
Comparing these results and consulting Table $2$, we find that $a = y = 1$ and $b = c = z = 0$.
Thus,
\[Sq^3 \overline{ \alpha_1 } + Sq^2 \overline{ \alpha_2 } = x Sq^2 \gamma_2 + Sq^1 \gamma_3\]
for some $x \in \Z/2$.
Since $\pi_{\varPsi}^{\ast}(\gamma_2) =0$,
the proof is complete.
\end{proof}

%\begin{remark}
%It follows from the proof of Lemma \ref{lem:relation} that we always have
%\begin{align}\label{eq:sq5}

%\end{align}
%Based on Lemma \ref{lem:relation}, $\overline { \alpha_2 }$ will be selected such that
%\begin{align*}
%Sq^3 \overline{ \alpha_1} + Sq^2 \overline{ \alpha_2}  & =  Sq^1 \gamma_3. 
%\end{align*}
Applying $Sq^1$, $Sq^2$ and $Sq^4$ to both sides of \eqref{eq:sq3a1} and using the Adem relations $Sq^1 Sq^{2k} = Sq^{2k+1}$, $Sq^{1} Sq^{2k+1} = 0$, $Sq^2 Sq^2 = Sq^3 Sq^1$, $Sq^2 Sq^3 = Sq^5 + Sq^4 Sq^1$, and $Sq^4 Sq^3 = Sq^5 Sq^2$, we obtain
\begin{align}
Sq^{3} \overline{ \alpha_2 } & = 0,\label{eq:sq3a2}\\
Sq^3 Sq^1 \overline{ \alpha_2 } & = Sq^5 \overline{ \alpha_1 } +  Sq^4 \gamma_2 + Sq^2 Sq^1 \gamma_3,\\
Sq^4 Sq^2 \overline{ \alpha_2 } & = Sq^5 Sq^2 \overline{\alpha_1} +  Sq^4 Sq^1 \gamma_3.
\end{align}
%\end{corollary}
%\begin{proof}
Using the Adem relations and \eqref{eq:sq51g3}, we further derive:
\begin{align}
Sq^4 Sq^1 \overline{ \alpha_2 } & = Sq^5 \overline{ \alpha_2 }, \\
Sq^5 Sq^1 \overline{ \alpha_2 }  & = 0,\\
Sq^5 Sq^2 \overline{ \alpha_2 } & =  Sq^5 Sq^1 \gamma_3 = 0.\label{eq:sq52a2}
\end{align}
%\end{proof}
%\end{remark}
%\begin{align}
%Sq^5 Sq^1 \overline{ Sq^{4n+2} l_r }  & = 0. 
%%Sq^6 Sq^1 \overline{ Sq^{4n+1} l_r }  & = Sq^2 ( \delta \sigma l_{4n} Sq^5 l_{4n} + \delta \sigma Sq^1 l_{4n} Sq^4 l_{4n} + \delta \sigma Sq^2 l_{4n} Sq^3 l_{4n}.)
%\end{align}
% set $\delta_n = 0$ if $n$ is even, and $\delta_n = 1$ if $n$ is odd, since 
%\begin{align*}
%Sq^4 Sq^{4n+4} & = \delta_n Sq^{4n+8} + Sq^{4n+7} Sq^1 + Sq^{4n+6} Sq^2, \\
%Sq^7 Sq^{4n+1} & = (1 {+} \delta_n) Sq^{4n+7} Sq^1 + Sq^{4n+5} Sq^3, \\
%Sq^6 Sq^{4n+2} & = \delta_n Sq^{4n+7} Sq^1 + Sq^{4n+6} Sq^2 + Sq^{4n+5} Sq^3,
%\end{align*}
%by the Adem relation \eqref{eq:adem},
Additionally, from \eqref{eq:a4} and \eqref{eq:sq5a4}, we have:
\begin{align}
%Sq^1 \alpha_1  & = 0, \\
%Sq^2 \alpha_2  & = Sq^3 \alpha_1, \\
Sq^4 \overline{ \alpha_4 }  + Sq^7 \overline{ \alpha_1 } + Sq^6 \overline{ \alpha_2 } +  \delta_m \overline{ \alpha_8 } \in (\Im \delta)^{+8}, \label{eq:sq4a4}\\
Sq^5 \overline{ \alpha_4 }  + Sq^7 \overline{ \alpha_2 } + \delta_m Sq^1 \overline{ \alpha_8 } \in (\Im \delta)^{+9},\label{eq:sq5a4b}
\end{align}

\subsection{Mod $2$ Cohomology Groups of $M\mathrm{Spin}^c(8s)$}
\label{ss:mspin}

For large $s$, let $T\colon H^{\ast}(B\mathrm{Spin}^c(8s)) \to H^{\ast}(M\mathrm{Spin}^c(8s)) $ be the Thom isomorphism, and let $U = T(1)$ be the Thom class.
%\begin{align*}
%H^{\ast}(B\mathrm{Spin};\Z/2) & = \Z/2[w_i ~|~i \neq 1, 2^r+1; r\ge0 ] \\
%& = \Z/2[w_{2j}, Sq^1 w_{2j}, V_{2^i} ~|~ \text{$j$ is not a power of $2$}, i\ge2 ].
%\end{align*}
%Since $B\mathrm{Spin}^c$ homotopy equivalent to $B\mathrm{Spin} \times BS^1$, 
%it follows that 
%\begin{align*}
%
%%& = \Z/2[w_{2j}, Sq^1 w_{2j}, V_{2^i} ~|~ \text{$j$ is not a power of $2$}, i\ge1 ].
%\end{align*}
Then $H^{\ast}(M\mathrm{Spin}^c(8s))$ is a free $H^{\ast}(B\mathrm{Spin}^c(8s);\Z/2)$-module generated by $U$.
Since
\[ H^{\ast}(B\mathrm{Spin}^c)  = \Z/2[w_i ~|~i \neq 1, 2^r+1; r\ge 1 ], \]
the definition of Stiefel-Whitney classes implies that
\begin{equation}\label{eq:sq1u}
 Sq^1 U = Sq^3 U = Sq^5 U = 0,
\end{equation}
and hence
\begin{equation}\label{eq:sq52u}
 Sq^5 Sq^2 U = Sq^4 Sq^3 U = 0.
\end{equation}
%Sq^3 Sq^4 U = Sq^7 U, \\
%Sq^2 Sq^4 U = Sq^6 U + Sq^5 Sq^1 U = Sq^6 U.
%\end{align*}
%\vspace{5pt}
%\begin{tabular}{|c|c|l|} \multicolumn{3}{c} { $ \bmod ~2$ cohomology groups of $MSpin^c(8s)$ } \\
%\hline
%\multicolumn{1}{|c}{\rule{0pt}{14pt} $j$} &
%\multicolumn{1}{|c}{$ \widetilde{H}^{8s+j}(MSpin^c(8s)) $} &
%\multicolumn{1}{|c|}{generators of $\widetilde{H}^{8s+j}(MSpin^c(8s))$}  \\
%%$j$ & $ \widetilde{H}^{r+4n+j}(C_{\varPsi}; \Z/2) $ & generators of $\widetilde{H}^{r+4n+j}(C_{\varPsi}; \Z/2)$ \\
%\hline
%\rule{0pt}{14pt} $0$ & $\Z/2$ & $U$ \\
%\hline
%\rule{0pt}{14pt} $2$ & $ \Z/2$ & $w_2 U$ \\
%\hline
%\rule{0pt}{14pt} $4$ & $ (\Z/2)^2$ & $w_2^2 U$, $w_4 U$ \\
%\hline
%\rule{0pt}{14pt} $6$ & $ (\Z/2)^3$ & $w_2^3 U$, $w_2 w_4 U$, $w_6 U$ \\
%\hline
%\rule{0pt}{14pt} $7$ & $ \Z/2$ & $w_7 U$ \\
%\hline
%\rule{0pt}{14pt} $8$ & $(\Z/2)^5$ & $w_2^4 U$, $w_2^2 w_4 U$, $w_2 w_6 U$, $w_4^2 U $, $w_8 U $ \\
%\hline
%\rule{0pt}{14pt} $1,3,5$ & $0$ &  \\
%\hline
%\end{tabular}
%\vspace{5pt}
Define $ U_4 = w_2^2 U$, $U_{81} = w_2^4 U$, and  $U_{82} = w_4^2 U$.
Then:
\begin{equation}\label{eq:sq1u4}
 Sq^1 U_4 = Sq^3 U_4 = 0.
\end{equation}
Through dimension $8s+8$, $H^{\ast}(M \mathrm{Spin}^c)$ is an $\mathscr{A}$-module generated by $U$, $U_4$, $U_{81}$, and $U_{82}$, subject to relations \eqref{eq:sq1u}-\eqref{eq:sq1u4}.
The basis of $H^{\ast}(M \mathrm{Spin}^c)$ through dimension $8s+8$ is listed in Table $4$.

\vspace{5pt}
\begin{center}
\begin{tabular}{|c|c|l|} \multicolumn{3}{c} {Table $4$. Mod $2$ cohomology groups of $M\mathrm{Spin}^c(8s)$ } \\
\hline
\multicolumn{1}{|c}{\rule{0pt}{14pt} $j$} &
\multicolumn{1}{|c}{$ \widetilde{H}^{8s+j}(M\mathrm{Spin}^c(8s)) $} &
\multicolumn{1}{|c|}{basis}  \\
%$j$ & $ \widetilde{H}^{r+4n+j}(C_{\varPsi}; \Z/2) $ & generators of $\widetilde{H}^{r+4n+j}(C_{\varPsi}; \Z/2)$ \\
\hline
\rule{0pt}{14pt} $1,3,5$ & $0$ &  \\
\hline
\rule{0pt}{14pt} $0$ & $\Z/2$ & $U$ \\
\hline
\rule{0pt}{14pt} $2$ & $ \Z/2$ & $Sq^2 U$ \\
\hline
\rule{0pt}{14pt} $4$ & $ (\Z/2)^2$ & $Sq^4 U$, $U_4$ \\
\hline
\rule{0pt}{14pt} $6$ & $ (\Z/2)^3$ & $Sq^6 U$, $Sq^4 Sq^2 U$,  $Sq^2  U_4 $ \\
\hline
\rule{0pt}{14pt} $7$ & $ \Z/2$ & $Sq^7 U$ \\
\hline
\rule{0pt}{14pt} $8$ & $(\Z/2)^5$ & $Sq^8 U $, $Sq^6 Sq^2 U$, $Sq^4  U_4$, $U_{81}$, $U_{82} $ \\
\hline
\end{tabular}
\end{center}
\vspace{5pt}

\subsection{Spin$^c$ Bordism Groups of $C_{\varPsi}$}
\label{ss:borCPsi}
Recall that $\M = M\mathrm{Spin}^c(8s) \wedge C_{\varPsi}$.
To simplify notation, let $K(G, +j)$ denote the Eilenberg-MacLane space of type $(G, r {+} 8s {+} 4m {+} j)$ for $1 \le j \le 9$, and let $l_{+j}$ denote the fundamental class of  $K(\Z/2, +j)$.

Since $C_{\varPsi}$ is $(r{+}4m)$-connected and $M\mathrm{Spin}^c(8s)$ is $(8s{-}1)$-connected, $\M$ is $(r+8s +4m)$-connected.
The reduced K\"unneth formula gives
\begin{align}\label{eq:calM}
H^{r + 8s + 4m + i} (\M) = \bigoplus_{j=0}^{i-1} H^{8s+j}(M\mathrm{Spin}^c(8s)) \tensor H^{r+ 4m + i - j }(C_{\varPsi}).
\end{align}
Combining Tables $3$ and $4$ with this formula, the cohomology groups $H^{r + 8s + 4m + i} (\M)$ for $i \le 9$ can be determined.

We now construct maps from $\M$ to Eilenberg-MacLane spaces  $K(G_i, +i)$ for $1 \le i \le 8$ to determine $\Omega_{r+4m+7}^{\mathrm{Spin}^c} (C_{\varPsi}) \cong \pi_{r+8s+4m+7}(\M)$, where the groups $G_i$ for $1 \le i \le 8$ are:
\begin{center}
\begin{tabular}{|c|c|c|c|c|c|c|c|c|} \multicolumn{9}{c} {} \\
\hline
\multicolumn{1}{|c}{\rule{0pt}{14pt} $i$} &
\multicolumn{1}{|c}{$ 1$} &
\multicolumn{1}{|c}{$ 2$} &
\multicolumn{1}{|c}{$ 3$} &
\multicolumn{1}{|c}{$ 4$} &
\multicolumn{1}{|c}{$ 5$} &
\multicolumn{1}{|c}{$ 6$} &
\multicolumn{1}{|c}{$ 7$} &
\multicolumn{1}{|c|}{$ 8$}  \\
%$j$ & $ \widetilde{H}^{r+4n+j}(C_{\varPsi}; \Z/2) $ & generators of $\widetilde{H}^{r+4n+j}(C_{\varPsi}; \Z/2)$ \\
\hline
\rule{0pt}{14pt} $G_i$ & $\Z/2$ & $\Z/2$ & $(\Z/2)^2$ &  $(\Z/2)^2$ & $(\Z/2)^4$ & $(\Z/2)^5$ & $(\Z/2)^8$ & $(\Z/2)^9$ \\
\hline
\end{tabular}
\end{center}
\vspace{5pt}

Define the following maps:
\begin{enumerate}
\item[(1)] $f_1 \colon  \M \to K(\Z/2, +1)$ satisfying
\[ f_1^{\ast} ( l_{+1} )  = U  \cdot \overline{\alpha_1},\]

\item[(2)] $f_2 \colon  \M \to K(\Z/2, +2)$ satisfying 
\[ f_2^{\ast} ( l_{+2})  = U  \cdot \overline{\alpha_2}.\]

\item[(3)] $f_3 = f_{31} \times f_{32} \colon \M \to K((\Z/2)^2, +3)$ the product map of $f_{31}$ and $f_{32}$,
where 
\[ f_{3j} \colon  \M \to K(\Z/2, +3),~j=1,2,\]
are the maps satisfying
\begin{align*}
f_{31}^{\ast} ( l_{+3} )  & = Sq^2 U  \cdot \overline{\alpha_1}, \\
f_{32}^{\ast} ( l_{+3} )  & = U  \cdot \gamma_3,
\end{align*}

\item[(4)] $f_4 = f_{41}\times f_{42} \colon \M \to K((\Z/2)^2, +4)$ the product map of $f_{41}$ and $f_{42}$,
where 
\[ f_{4j} \colon  \M \to K(\Z/2, +4), ~j=1,2, \]
are the maps satisfying
\begin{align*}
f_{41}^{\ast} ( l_{+4} ) & = U  \cdot \overline{\alpha_4}, \\
f_{42}^{\ast} ( l_{+4} ) & = U  \cdot \gamma_4,
\end{align*}

\item[(5)] $f_5 = \prod_{j=1}^{4} f_{5j} \colon \M \to K((\Z/2)^4, +5)$ the product map of $f_{5j}$, $1\le j\le 4$,
where
\[ f_{5j} \colon  \M \to K(\Z/2, +5), ~1\le j \le 4,\]
are the maps satisfying
\begin{align*}
f_{51}^{\ast} (l_{+5}) & = Sq^4 U  \cdot \overline{\alpha_1}; \\
f_{52}^{\ast} (l_{+5}) & = U_4  \cdot \overline{\alpha_1}; \\
f_{53}^{\ast} (l_{+5}) & = U  \cdot \gamma_5; \\
f_{54}^{\ast} (l_{+5}) & = Sq^2 U  \cdot \gamma_3.
\end{align*}

\item[(6)] $f_6 = \prod_{j=1}^{5} f_{6j} \colon \M \to K((\Z/2)^5, +6)$ the product map of $f_{6j}$, $1\le j\le 5$,
where
\[ f_{6j} \colon  \M \to K(\Z/2, +6), ~1\le j \le 5,\]
are the maps satisfying
\begin{align*}
f_{61}^{\ast} ( l_{+6} ) & = Sq^4 U  \cdot \overline{\alpha_2} ; \\
f_{62}^{\ast} ( l_{+6} ) & = U_4  \cdot \overline{\alpha_2}; \\ 
f_{63}^{\ast} ( l_{+6} ) & = Sq^2 U  \cdot \overline{\alpha_4}; \\
f_{64}^{\ast} ( l_{+6} ) & = U  \cdot \gamma_6; \\
f_{65}^{\ast} ( l_{+6} ) & = Sq^2 U  \cdot \gamma_4.
\end{align*}

\item[(7)] $f_7 = \prod_{j=1}^{8} f_{7j} \colon \M \to K((\Z/2)^8, +7)$ the product map of $f_{7j}$, $1\le j\le 8$,
where
\[ f_{7j} \colon  \M \to K(\Z/2, +7), ~1\le j \le 8,\]
are the maps satisfying
\begin{align*}
f_{71}^{\ast} ( l_{+7} ) & = U  \cdot Sq^6 \overline{\alpha_1} ; \\
f_{72}^{\ast} ( l_{+7} ) & = U  \cdot Sq^4 Sq^2 \overline{\alpha_1}; \\ 
f_{73}^{\ast} ( l_{+7} ) & = U_4  \cdot  Sq^2 \overline{\alpha_1}; \\
f_{74}^{\ast} ( l_{+7} ) & = U  \cdot Sq^4 \gamma_3; \\
f_{75}^{\ast} ( l_{+7} ) & = U  \cdot Sq^2 \gamma_5; \\
f_{76}^{\ast} ( l_{+7} ) & = U  \cdot \gamma_7; \\
f_{77}^{\ast} ( l_{+7} ) & = U  \cdot \gamma_{71}; \\
f_{78}^{\ast} ( l_{+7} ) & = U_4  \cdot \gamma_3.
\end{align*}

\item[(8)] $f_8 = \prod_{j=1}^{9} f_{8j} \colon \M \to K((\Z/2)^9, +8)$ the product map of $f_{8j}$, $1\le j\le 9$,
where
\[ f_{8j} \colon  \M \to K(\Z/2, +8), ~1\le j \le 9,\]
are the maps satisfying
\begin{align*}
f_{81}^{\ast} ( l_{+8} ) & = U  \cdot Sq^6 \overline{\alpha_2} ; \\
f_{82}^{\ast} ( l_{+8} ) & = U  \cdot  \overline{\alpha_8}; \\ 
f_{83}^{\ast} ( l_{+8} ) & = U_4  \cdot  \overline{\alpha_4}; \\
f_{84}^{\ast} ( l_{+8} ) & = U  \cdot Sq^4 Sq^2 \gamma_2; \\
f_{85}^{\ast} ( l_{+8} ) & = U  \cdot Sq^4 \gamma_4; \\
f_{86}^{\ast} ( l_{+8} ) & = U  \cdot Sq^2 \gamma_6; \\
f_{87}^{\ast} ( l_{+8} ) & = U  \cdot \gamma_{8}; \\
f_{88}^{\ast} ( l_{+8} ) & = U  \cdot \gamma_{81}; \\
f_{89}^{\ast} ( l_{+8} ) & = U_4  \cdot \gamma_4.
\end{align*}
\end{enumerate}
Now, let
\begin{align*}
f = \prod_{i=1}^{8}f_{i} \colon \M \to \prod_{i=1}^{8} K(G_i, +i)
\end{align*}
be the product map of $f_{i}$, $1\le i\le 8$.
Let $K = \prod_{i=1}^{8}K(G_i, +i)$, and let $F$ be the fiber of $f$, giving the fibration
\begin{align*}
F \hookrightarrow \M \xrightarrow{f} K.
\end{align*}

\begin{lemma} \label{lem:Fco}
$F$ is $(r+8s+4m-1)$-connected.
\end{lemma}
\begin{proof}
This follows because both $\M$ and $K$ are $(r+8s +4m)$-connected.
\end{proof} 

Let $p_i \colon K \to K(G_i, +i)$ the projection such that $p_i \circ f = f_i$ for $1 \le i \le 8$.
For $3 \le i \le 8$, let $p_{ij} \colon K(G_i, +i) \to K(\Z/2, +i)$ be the map such that $p_{ij} \circ f_{i} = f_{ij}$ for suitable $j$.
Denote $p_{1}^{\ast}(l_{+1})$ and $p_{2}^{\ast}(l_{+2})$ simply as $l_{+1}$ and $l_{+2}$.
Define: 
\begin{align*}
l_{+3,1} = p_3^{\ast} \circ p_{31}^{\ast} (l_{+3}),\\
l_{+3,2} = p_3^{\ast} \circ p_{32}^{\ast} (l_{+3}),\\
l_{+5,4} = p_5^{\ast} \circ p_{54}^{\ast} (l_{+5}).
\end{align*}
Let $\xi \in H^{r+8s+4m+6}(K)$ be defined as
\begin{align*}
\xi := & ~ Sq^5 l_{+1} + Sq^4 Sq^1 l_{+1} + Sq^3 Sq^1 l_{+2} + Sq^3 l_{+3, 1} + Sq^2 Sq^1 l_{+3, 1} + Sq^2 Sq^1 l_{+3, 2} + Sq^1 l_{+5, 4}.
\end{align*}

\begin{lemma}\label{lem:fxi}
Suppose $m \ge 2$. 
For large $r$ and $s$, the induced homomorphism
\[ f^{\ast} \colon H^{r+8s+4m+j}(K) \to H^{r+8s+4m+j}(\mathcal{M}) \]
is an epimorphism for $j \le 8$.
Through dimension $r + 8s + 4m +9$ the kernel of $f^{\ast}$ is generated over the Steenrod algebra $\mathscr{A}$ by $\xi$.
\end{lemma}
\begin{proof}
%Note that $r$ and $s$ are large enough.
Since $F$ is $(r+8s+4m-1)$-connected (Lemma \ref{lem:Fco}) and $K$ is $(r+8s+4m)$-connected, the Serre long exact cohomology sequence (Lemma \ref{lem:serre}) gives:
\begin{align*}
H^1(\mathcal{M}) \to \cdots \to H^j(F) \xrightarrow{\tau} H^{j+1}(K) \xrightarrow{f^{\ast}} H^{j+1}(\mathcal{M}) \to H^{j+1}(F) \xrightarrow{\tau} \cdots  \to H^{2r+16s+8m}(F),
\end{align*}
where $\tau$ is the transgression.
The basis of $H^{\ast}(\mathcal{M})$ through dimension $r+8s+4m+9$ is determined by \eqref{eq:calM} and Tables $3$ and $4$,
and the $\mathscr{A}$-module relations it satisfies are given by Tables $3$ and $4$, Lemmas \ref{lem:sq1a1} and \ref{lem:sq3a1}, and Identities \eqref{eq:sq1g2}-\eqref{eq:sq51g3}, \eqref{eq:sq3a2}-\eqref{eq:sq52a2}, \eqref{eq:sq1u}-\eqref{eq:sq1u4}, and Relations \eqref{eq:sq4a4} and \eqref{eq:sq5a4b}.
Combining this with Lemma \ref{lem:mac} and the construction of $f$, the result follows from a straightforward (though tedious) calculation of $f^{\ast}$ using the Serre long exact cohomology sequence above. 
\end{proof}
%By  calculation, one may then verify that $f$ induces an epimorphism in $\bmod{2}$ cohomology through dimension $8s + r + 4n +8$, and that through dimension $8s + r + 4n +9$ the kernel is generated over the Steenrod algebra $\mathscr{A}$ by $\xi$, where
%\begin{align*}
%\xi = & ~ Sq^5 (U \cdot \overline{\alpha_1} ) + Sq^4 Sq^1 (U \cdot \overline{\alpha_1} ) + Sq^3 Sq^1 ( U \cdot \overline{\alpha_2} ) \\
%& + Sq^3 (Sq^2 U \cdot \overline{\alpha_1} ) + Sq^2 Sq^1 (Sq^2 U \cdot \overline{\alpha_1} ) + Sq^2 Sq^1 ( U \cdot \gamma_3) + Sq^1 (Sq^2 U \cdot \gamma_3).
%\end{align*}

%\[Sq^3 Sq^1 \xi = 0.\]
%
%One may then choose a minimal set of additional generators in dimension $8s + r + 4n + 9$, giving 
%\begin{align*}
%\hat{f} \colon \M \to \prod_{i=1}^{9} K(G_i, +i)
%\end{align*}
%so that $\hat{f}^{\ast}$ is epic through dimension $8s + r + 4n + 9$, and has kernel generated by $\xi$ over $\mathscr{A}$ through this dimension.
%
%Letting $F$ be the fiber of $\hat{f}$, one then has a fibration
%\begin{align*}
%F \hookrightarrow \M \xrightarrow{\hat{f}} \prod_{i=1}^{9} K(G_i, +i)
%\end{align*}
%and may calculate

\begin{theorem}\label{thm:borCPsi}
Suppose $m \ge 2$. 
For large $r$ and $s$, the induced homomorphism
\[ f_{\ast} \colon \pi_{r+8s+4m+j}(\mathcal{M}) \to \pi_{r+8s+4m+j}(K) \]
is an isomorphism for $j \le 4$ and $j = 7$.
\end{theorem}

\begin{proof}
Let $\ell_{+5} \in H^{r+8s+4m+5}(F)$ be the element such that $\tau(\ell_{+5}) = \xi$, and let 
\[ e \colon F \to K(\Z/2, +5)\]
be the map satisfying $e^{\ast} ( l_{+5} ) = \ell_{+5}$.
By Lemmas \ref{lem:mac} and \ref{lem:fxi}, the induced homomorphism
\[ e^{\ast} \colon H^{r+8s+4m+j}( K(\Z/2, +5) )  \to H^{r+8s+4m+j} (F) \]
is an isomorphism for $j \le 8$.

Let $\hat{F}$ be the fiber of $e$, giving the fibration
\[ \hat{F} \hookrightarrow F \xrightarrow{e} K(\Z/2, +5). \] 
The homotopy groups of $\mathcal{M}$, $F$ and $\hat{F}$ are all purely $2$-primary.
Since $F$ is $(r+8s+4m-1)$-connected by Lemma \ref{lem:Fco} and $K(\Z/2, +5)$ is $(r+8s+4m+4)$-connected,
$\hat{F}$ is $(r+8s+4m-2)$-connected.
From the Serre long exact cohomology sequence for this fibration, we find:
\[ H^{r+8s+4m+j}(\hat{F}) = 0 \text{~for~}  j \le 7.\] 
Thus, $\hat{F}$ is $(r+8s+4m+7)$-connected and
\[e_{\ast} \colon \pi_{i}(F) \to \pi_{i}(K(\Z/2, +5))\]
is an isomorphism for $i \le r+8s+4m+7$. 
The theorem now follows by analyzing the long exact sequence of homotopy groups for the fibration $F \hookrightarrow \mathcal{M} \to K$.
%\remhy{
%Let $X$ be a path-wise connected $CW$ complex. 
%Suppose that $X$ is simply connected and all cohomology groups $H^{\ast}(X;\Z)$ of $X$ are torsion groups and only have $2$-torsion. 
%Then $H^{i}(X;\Z/2) =0$ for all $i \le n$ implies that $X$ is $n$-connected.
%}
\end{proof}

\subsection{Proof of Theorem \ref{thm:main}}
\label{ss:pfmain}

We now prove Theorem \ref{thm:main} using the results from Subsections \ref{ss:idea}-\ref{ss:borCPsi}.

For any positive integer $m$ and large $r$ and $s$, consider the following diagram:
\begin{equation}\label{diag:pf}
\begin{split}
\xymatrix{
 \wt{\Omega}^{\mathrm{Spin}^c}_{8m+7}( K(\Z/2, 4m)) \ar[r]^-{\varPsi_{\ast}} & \wt{\Omega}^{\mathrm{Spin}^c}_{r+4m+7}(K(\Z/2, r)) \ar[r]^-{\pi_{\varPsi\ast}} \ar[d]_{\mathcal{P}_2}^{\cong} &  \wt{\Omega}^{\mathrm{Spin}^c}_{r+4m+7}(C_{\varPsi})  \ar[d]_-{PT}^{\cong}  \\
 & H_{4m+7}(B\mathrm{Spin}^c) & \pi_{r+8s+4m+7}(\mathcal{M}) \ar[d]_{f_{7\ast}}^{\cong} \\
 & & \pi_{r+8s+4m+7}(K((\Z/2)^8, +7)).
}
\end{split}
\end{equation}
Here, the horizontal sequence is exact, $\mathcal{P}_2$ is the isomorphism from Lemma \ref{lem:oh2}, $PT$ is the Pontrjagin-Thom isomorphism, and $f_{7\ast}$ is induced by the map $f_{7}$ constructed in Subsection \ref{ss:borCPsi}.
By Theorem \ref{thm:borCPsi}, $f_{7\ast}$ is an isomorphism.
%where the homomorphism $\varPsi_{\ast}$ is given by 
%\begin{equation*}
%\varPsi_{\ast}([N,f])=\tau_{N\ast}([N]\cap f^{\ast} (l_{4m})),
%\end{equation*}
%for any bordism class $[N, f]\in  \wt{\Omega}^{\mathrm{Spin}^{c}}_{8m+7}(K(\Z/2, 4m))$ with $f\colon N \rightarrow K(\Z/2, 4m)$.
%Note that $\Omega_{r+4m+7}^{\mathrm{Spin}^c} (C_{\varPsi}) \cong \pi_{r+8s+4m+7}(M\mathrm{Spin}^c(8s) \wedge C_{\varPsi})$ for large $s$.

To prove Theorem \ref{thm:main}, we need the following lemmas.

For any $y \in H^i(B\mathrm{Spin}^c(8s))$ and $z \in H^{r+4m+7-i}(C_{\varPsi})$, let
\[ f_{yz} \colon M\mathrm{Spin}^c(8s) \wedge C_{\varPsi} \to K(\Z/2, +7)\]
be the map satisfying 
\[ f_{yz}^{\ast} ( l_{+7} ) =  U \cdot y \cdot z \in H^{r+8s+4m+7}(M\mathrm{Spin}^c(8s) \wedge C_{\varPsi}),\]
where $U$ is the Thom class.

For any $[N, f] \in \wt{\Omega}_{r+4m+7}^{\mathrm{Spin}^c}(K(\Z/2, r))$,
let 
\[ \phi \colon S^{r+8s+4m+7} \to K(\Z/2, +7) \]
represent the element 
\[ f_{yz \ast} \circ PT \circ \pi_{\varPsi \ast} ( [N, f]) \in \pi_{r+8s+4m+7} ( K(\Z/2, +7) ), \]
and let $[S]$ be the fundamental class of $S^{r+8s+4m+7}$.
\begin{lemma}\label{lem:nf}
The element $f_{yz \ast} \circ PT \circ \pi_{\varPsi \ast} ( [N, f])$ is detected by $\langle \phi^{\ast} ( l_{+7} ), [S] \rangle$ and
\[ \langle \phi^{\ast} ( l_{+7} ), [S] \rangle = \langle \tau_{N}^{\ast} (y) \cdot f^{\ast} ( \pi_{\varPsi}^{\ast} (z) ), [N] \rangle.\]
\qed
\end{lemma}
%Now, take $m = n{-}1$, it follows from Lemma \ref{lem:zero} and Equation \eqref{eq:Psi} that 
%\begin{equation*}
%\langle Sq^1 \Theta, x \rangle = 0, 
%\end{equation*}
%for any $x \in \mathrm{Im} \varPsi_{\ast}$.

Regarding the Wu class, we have:
%\begin{lemma}\label{lem:wu}
%Let $v_i \in H^{i}(B\mathrm{Spin}^c)$ be the $i$-th universal Wu class. 
%Then $v_{2k+1} = 0$ for any non-negative integer $k$, $v_2 = w_2$, $v_6 = w_2 \cdot w_4$ and hence $Sq^1 v_6 = 0$.
%\end{lemma}
%\begin{proof}
%It followsthat $v_1 = w_1 = 0$, $v_2 = w_2 $.
%Hence $v_{2k+1} = 0$ for any non-negetive integer $k$ and 
%\begin{align*}
%v_4 & = w_4 + w_2^2,\\
%v_6 & = w_6 + Sq^2 v_{4}
%\end{align*}
%Now since $Sq^1 w_2= 0$,  it follows from  Wu's formula \eqref{eq:wu}  Wu's explicit formula \cite[p. 94, Problem 8-A]{ms74} that 
%$Sq^2 v_4 = Sq^2 w_4+ Sq^2 w_2^2 = Sq^2 w_4(N) = w_2w_4 + w_6$, 
%and hence , which completes the proof.
%\end{proof}
\begin{lemma}\label{lem:v2k}
For any $n$-dimensional spin$^c$ manifold $N$ and any nonnegative integer $k$,
%Let $v_{4n} \in H^{4n}(B\mathrm{Spin}^c)$ be the $4n$-th universal Wu class. Then we always have
\[ Sq^1 v_{2k}(N) = 0.\]
\end{lemma}

%\begin{remark}
%Since all torsion in $H^{\ast}(B\mathrm{Spin}^c;\Z)$ has order $2$ (cf. \cite[p. 317, Corollary]{st68b}), it can be deduced easily from $Sq^1 v_{4n} = 0$ that $\beta^{\Z/2} (v_{4n}) = 0$.
%\end{remark}

\begin{proof}
If $n \le 2k+1$ or $k \le 1$, the identity holds trivially.

Assume $n > 2k+1$ and $k \ge 2$. 
By Poincar\'e Duality Theorem, it suffices to show that $\langle Sq^1 v_{2k}(N) \cdot x, [N] \rangle = 0$ for any $x \in H^{n-2k-1}(N)$.
Since $v_1(N) = 0$, $v_{2k+1}(N) =0$, $Sq^1 v_2(N) = 0$, and 
\[ Sq^2 Sq^{2k-1} =  \binom{2k-2}{2} Sq^{2k+1} + Sq^{2k} Sq^1\]
by the Adem relation \eqref{eq:adem}, 
we have
\begin{align*}
\langle Sq^1 v_{2k}(N) \cdot x, [N] \rangle & = \langle v_{2k}(N) \cdot Sq^1 x, [N]\rangle \\
& = \langle Sq^{2k} Sq^1 x, [N]\rangle \\
& = \langle Sq^2 Sq^{2k-1} x, [N] \rangle \\
& = \langle v_2(N) \cdot Sq^1 Sq^{2k-2} x, [N] \rangle \\
& = \langle Sq^1 v_2(N) \cdot Sq^{2k-2} x, [N] \rangle\\
& =0.
\end{align*}
\end{proof}

%Let $\chi$ be the canonical anti-automorphism of the $\bmod ~2$ Steenrod algebra $\mathscr{A}$ (cf. Milnor \cite{mi-std}).
%For large $s$, denote by $T\colon H^{\ast}(B\mathrm{Spin}^c(8s)) \to H^{\ast}(M\mathrm{Spin}^c(8s)) $ the Thom isomorphism, and $U = T(1)$ the Thom class.
%%Let $Sq = \Sigma_{i=0}^{\infty} Sq^i$. 
%%We known that $\chi(Sq) = Sq^{-1}$.
%We known that $v_{4n}$ is defined by 
%\begin{equation} \label{eq:v4n}
%T (v_{4n}) = \chi(Sq^{4n})U.
%\end{equation}
% Let $Q_0 = Sq^1$ and $Q_1 = Sq^2 Sq^1 + Sq^1 Sq^2$. 
%A trivial verification shows that 
%\begin{equation}\label{eq:sq4n}
%Sq^{4n} Q_0 = Q_0 Sq^{4n} + Q_1 Sq^{4n-2}
%\end{equation}
%by the Adem relation \eqref{eq:adem}. 
%Note that we have $\chi(Q_0) = Q_0$ and $\chi(Q_1) = Q_1$ (see \cite{mi-std}).
%Then apply $\chi$ to the identity \eqref{eq:sq4n}, one get that 
%\begin{equation} \label{eq:chi4n}
% Q_0 \chi(Sq^{4n}) = \chi(Sq^{4n}) Q_0 + \chi(Sq^{4n-2}) Q_1.
% \end{equation}
%Now, since $Q_0 U = Sq^1 U = 0$ and hence $Q_0 = Sq^1$ commutes with the Thom isomorphism $T$, it follows from $Q_0 w_2 = 0$ that $Q_1 U = Q_0 T (w_2) = T(Q_0 w_2) = 0$.
%Here, $w_2 \in H^2(B\mathrm{Spin}^c)$ is the second universal Stiefel-Whitney class.
%Then, combing the definition of $v_{4n}$ \eqref{eq:v4n} with the identity \eqref{eq:chi4n}, we obtain
%\begin{align*}
%T(Sq^1 v_{4n}) = Sq^1 T (v_{4n}) = Q_0 \chi(Sq^{4n}) U = \chi(Sq^{4n}) Q_0 U  + \chi(Sq^{4n-2}) Q_1 U = 0,
%\end{align*}
%which completes the proof.

We now prove Theorem \ref{thm:main}.

\begin{proof}[Proof of Theorem \ref{thm:main}]
For $n=1$, 
\[ H^{6}(B\mathrm{Spin}^c)/\rho_2(H^{6}(B\mathrm{Spin}^c; \Z)) \cong \Z/2 \]
generated by $Sq^2 v_4$.
By Proposition \ref{prop:theta0}, $\Theta = Sq^2 v_4$, and Theorem \ref{thm:theta} completes the proof for $n=1$.

Now assume $n \ge 3$.
Set $m = n-1$.
Using the notation from Subsections \ref{ss:cofi} and \ref{ss:borCPsi}, we first determine the image of $\pi_{\varPsi\ast}$ in Diagram \eqref{diag:pf},
which is equivalent to determining the image of $ f_{7\ast} \circ PT \circ \pi_{\varPsi\ast}$.
Based on the construction of $f_{7}$, we compute $f_{7i\ast} \circ PT \circ \pi_{\varPsi\ast}$ for $1 \le i \le 8$ and any bordism class $[N, f] \in \wt{\Omega}_{r+4m+7}^{\mathrm{Spin}^c}(K(\Z/2, r))$.

For $4 \le i \le 8$, since $\pi_{\varPsi}^{\ast} (\gamma_j) = \pi_{\varPsi}^{\ast} (\gamma_{71}) =0$
for $j = 3, 5, 7$ (Table $2$), Lemma \ref{lem:nf} implies: 
\[ f_{7i\ast} \circ PT \circ \pi_{\varPsi \ast} ( [N, f]) = 0.\]
%for any $[N, f] \in \wt{\Omega}_{r+4m+7}^{\mathrm{Spin}^c}(K(\Z/2, r))$.

For $i = 1$, $f_{71\ast} \circ PT \circ \pi_{\varPsi \ast} ( [N, f])$ is detected by
\begin{align*}
\langle Sq^6 f^{\ast} ( \alpha_1 ), [N] \rangle & = \langle Sq^6 Sq^{4m+1} f^{\ast} ( l_r ), [N] \rangle \\
& = \langle Sq^6 Sq^{4n-3} f^{\ast} ( l_r ), [N] \rangle \\
& = \langle v_6(N) \cdot Sq^{4n-3} f^{\ast} (  l_r ), [N] \rangle \\
& = \langle Sq^1 v_6(N) \cdot Sq^{4n-4} f^{\ast} ( l_r ) , [N] \rangle.
\end{align*}
By Wu's formula \eqref{eq:wu} and Wu's explicit formula \cite[p. 94, Problem 8-A]{ms74}, 
$v_6 = w_2 w_4$, so $Sq^1 v_6 = 0$.
Thus,
\begin{align*}
\langle Sq^6 f^{\ast} ( \alpha_1 ), [N] \rangle 
 = \langle Sq^1 v_6(N) \cdot Sq^{4n-4} f^{\ast} ( l_r ) , [N] \rangle = 0,
\end{align*}
and hence $f_{71\ast} \circ PT \circ \pi_{\varPsi \ast} ( [N, f]) = 0$.

For $i = 3$, $f_{73\ast} \circ PT \circ \pi_{\varPsi \ast} ( [N, f])$ is detected by 
\begin{align*}
\langle w_2^2(N) \cdot Sq^2 f^{\ast} ( \alpha_1 ), [N] \rangle & = \langle w_2^2(N)  \cdot Sq^2 Sq^{4m+1} f^{\ast} ( l_r ), [N] \rangle \\
& = \langle w_2^2(N)  \cdot Sq^2 Sq^{4n-3} f^{\ast} ( l_r ), [N] \rangle \\
& = \langle Sq^2 [ w_2^2(N)  \cdot  Sq^{4n-3} f^{\ast} ( l_r ) ], [N] \rangle \\
& = \langle w_2^3(N) \cdot Sq^{4n-3} f^{\ast} ( l_r ), [N]\rangle \\
& = \langle Sq^1 w_2^3(N) \cdot Sq^{4n-4} f^{\ast} ( l_r ), [N]\rangle \\
& = 0,
\end{align*}
so $f_{73\ast} \circ PT \circ \pi_{\varPsi \ast} ( [N, f]) = 0$.

%Thus, by Theorem \ref{thm:borCPsi} and the construction of the map $f$, the image of $\pi_{\varPsi\ast}$ is at most $\Z/2$ and is detected by $U \cdot Sq^4 Sq^2 \overline{\alpha_1}$.
%For any $x \in H_{4n+3}(B\mathrm{Spin}^c) = H_{4m+7}(B\mathrm{Spin}^c)$, let $[N, f] \in \wt{\Omega}_{r+4m+7}^{Spin^c}(K(\Z/2, r))$ be the cobordism class satisfying 
%\[ \mathcal{P}_2 ([N, f]) = x.\]
For $i=2$, $f_{72\ast} \circ PT \circ \pi_{\varPsi \ast} ( [N, f])$ is detected by:
\[  \langle Sq^4 Sq^2 f^{\ast} ( \alpha_1 ) , [N] \rangle  = \langle Sq^4 Sq^2 Sq^{4m+1} f^{\ast} ( l_r ), [N]\rangle 
 = \langle Sq^4 Sq^2 Sq^{4n-3} f^{\ast} ( l_r ), [N]\rangle. \]
By the Adem relation \eqref{eq:adem}, 
\[ Sq^4 Sq^2 Sq^{4n-3} = \binom{4n-3}{4}  Sq^{4n+2} Sq^1 + Sq^{4n} Sq^2 Sq^1. \]
Now,
 \begin{align*}
 \langle Sq^{4n+2} Sq^1 f^{\ast} ( l_r ) , [N] \rangle  & = \langle v_{4n+2}(N) \cdot Sq^1 f^{\ast} ( l_r ), [N]\rangle \\
 & = \langle Sq^1 v_{4n+2}(N) \cdot f^{\ast}(l_r), [N] \rangle\\
 & = 0
 \end{align*}
 by the definition of Wu classes and Lemma \ref{lem:v2k}.
 Therefore,
\begin{align*}
\langle Sq^4 Sq^2 f^{\ast} ( \alpha_1 ) , [N] \rangle & = \langle Sq^4 Sq^2 Sq^{4n-3} f^{\ast} ( l_r ), [N]\rangle \\
& =  \langle Sq^{4n} Sq^2 Sq^1 f^{\ast} ( l_r ) , [N]\rangle \\
%& = \langle (\delta_n + 1) Sq^{2} Sq^{4n+1} f^{\ast} ( l_r ) , [N] \rangle + \langle Sq^{4n} Sq^2 Sq^1 f^{\ast} ( l_r ) , [N]\rangle \\
& = \langle v_{4n}(N) \cdot Sq^2 Sq^1 f^{\ast} ( l_r ) , [N]\rangle \\
& = \langle v_2(N) \cdot v_{4n}(N) \cdot Sq^1 f^{\ast} ( l_r ) \rangle  + \langle Sq^2 v_{4n}(N) \cdot Sq^1 f^{\ast} ( l_r ) , [N]\rangle \\
& =  \langle Sq^1 [ v_2(N) \cdot v_{4n}(N)) ] \cdot f^{\ast} ( l_r )  + Sq^1 Sq^2 v_{4n}(N) \cdot f^{\ast} ( l_r ) , [N]\rangle\\
& = \langle Sq^1 Sq^2 v_{4n}(N) \cdot f^{\ast} ( l_r ) , [N]\rangle,
\end{align*}
where the last step uses Lemma \ref{lem:v2k}.
Since 
\begin{align*}
\langle Sq^1 Sq^2 v_{4n}(N) \cdot f^{\ast} ( l_r ) , [N]\rangle & = \langle Sq^1 Sq^2 v_{4n}, \tau_{N\ast} ( [N] \cap f^{\ast} (l_r) )\rangle \\
& = \langle Sq^1 Sq^2 v_{4n}, \mathcal{P}_2([N,f]) \rangle,
\end{align*}
and $\mathcal{P}_2$ is an isomorphism, the above calculations shows that the image of $\pi_{\varPsi\ast}$ is $\Z/2$ and is detected by $\langle Sq^1 Sq^2 v_{4n} , x \rangle$ for any $x \in H_{4n+3}(B\mathrm{Spin}^c)$.
By Lemmas \ref{lem:nozero} and \ref{lem:zero}, we conclude 
\[ Sq^1 \Theta = Sq^1 Sq^2 v_{4n}, \]
and thus $\Theta = Sq^2 v_{4n} $.
Theorem \ref{thm:theta} now completes the proof for $n \ge 3$.

For $n = 2$, we use the result for $n=3$.
Let $\mathbb{H} P^{2}$ be the quaternionic projective plane with generator $u \in H^{4}(\mathbb{H}P^2)$.
For any $18$-dimensional spin$^c$ manifold $M$, a direct calculation shows:
\[ v_{12}( M\times \mathbb{H}P^2 ) = v_8(M) \otimes \rho_2 (u). \]
For any torsion class $t \in TH^{8}(M;\Z)$, the result for $n=3$ gives 
\begin{align*}
\langle \rho_2(t ) \cdot Sq^2 \rho_2 (t), [M] \rangle & = \langle \rho_2(t \otimes u) \cdot Sq^2 \rho_2(t \otimes u), [M \times \mathbb{H}P^2] \rangle\\
& = \langle \rho_2(t \otimes u) \cdot Sq^2 v_8(M) \otimes \rho_2(u), [M \times \mathbb{H}P^2] \rangle \\
& = \langle \rho_2(t) \cdot Sq^2 v_8(M), [M] \rangle,
\end{align*}
completing the proof for $n=2$.
\end{proof}

\begin{proof}[Proof of Corollary \ref{coro:bsqv}]
Since $\Theta = Sq^2 v_{4n}$, the result follows immediately from Proposition \ref{prop:sqtheta}.
\end{proof}

%The calculation precess of this outcome is detailed in Appendix A.
%{s:rank}
%%%%%%%%%%%%%%%%%%%%%%%%%%%%%%%%%%%%%%%%%%%%%%%%%%%
%\section{Proof of Corollary \ref{coro:rank}}

%{s:cobordism}
%%%%%%%%%%%%%%%%%%%%%%%%%%%%%%%%%%%%%%%%%%%%%%%%%%%%

%{s:pre}
%%%%%%%%%%%%%%%%%%%%%%%%%%%%%%%%%%%%%%%%%%%%%%%%%%%

%\bibliography{map}
%\bibliographystyle{alpha}

%\bibliographystyle{amsinitial}
%\bibliography{...}

\bibliographystyle{amsplain}
%\bibliography{map}
%\bibliographystyle{amsinitial}
%\bibliography{HY}

%\begin{multicols}{2}[]
%\bigskip
%\noindent
%\emph{Diarmuid Crowley}\\
%  \vskip -0.125in \noindent
%  {\small
%  \begin{tabular}{l}%
%    School of Mathematics and Statistics\\
%    University of Melbourne\\
%    Parkville, VIC, 3010, Australia\\
%    \textsf{dcrowley@unimelb.edu.au}
%	%
%  \end{tabular}}
%
%
%\noindent
%\emph{Huijun Yang}\\  
%\vskip -0.125in \noindent
%{\small
%  \begin{tabular}{l}%
%   School of Mathematics and Statistics\\
%   Henan University\\
%   Kaifeng, Henan, 475004, China\\
%   %
%    \textsf{yhj@amss.ac.cn}
%    %
%  \end{tabular}}
%\end{multicols}

\end{document}